\documentclass[reqno,11pt]{amsart}
\usepackage{ytableau}
\usepackage{youngtab}
\usepackage[bbgreekl]{mathbbol} 
\DeclareSymbolFontAlphabet{\mathbbm}{bbold}
\DeclareSymbolFontAlphabet{\mathbb}{AMSb}
\usepackage{amsmath, amssymb, amsthm, amsfonts, graphicx, mathrsfs, cancel, color, tikz, mathtools, nccmath, needspace, comment, stmaryrd, array}
\usepackage{thmtools}
\usepackage[a4paper,margin=0.9in]{geometry}

\usepackage[pagebackref=true,bookmarks=true,colorlinks=true,linktoc=page,citecolor=darkgreen,linkcolor=blue,urlcolor=mediumblue,unicode]{hyperref}

\usepackage[shortlabels]{enumitem}

\newenvironment{enua}{\begin{enumerate}[label=\textup{(\alph*)}]
}{\end{enumerate}}

\newcommand{\define}[4]{\expandafter#1\csname#3#4\endcsname{#2{#4}}}
\forcsvlist{\define{\DeclareMathOperator}{}{}}{ad,Ad,codim,GL,SL,vdim} 
\forcsvlist{\define{\newcommand}{\mathbf}{b}}{A,C,D,H,S,T,v,w,Z,a,b,i,j,k,1,K} 
\forcsvlist{\define{\newcommand}{\mathbb}{bb}}{A,B,D,H,N,Z,F,Q,R,P,S,C,T,L,G,V} 
\forcsvlist{\define{\newcommand}{\mathcal}{c}}{A,B,C,D,E,F,G,H,I,J,K,L,M,N,O,P,Q,R,S,T,U,V,W,X,Y,Z} 
\forcsvlist{\define{\newcommand}{\mathfrak}{f}}{a,h,g,gl,H,I,p,q,s,t,u,S,sl,z} 
\forcsvlist{\define{\newcommand}{\mathrm}{}}{crit,id,Id,pt,op,sm,gdim,aff,Ann,Aut,asph,colim,End,ext,Ext,Frac,gr,Gr,Hom,Ind,Irr,Ker,KZ,Mat,mon,Mod,Proj,rad,Res,RHom,rk,Span,sph,Sch,Stab,sgn,Tor,Tr,triv,ind,type} 
\forcsvlist{\define{\DeclareMathOperator}{\mathsf}{}}{BM,Bun,Coh,Coker,Cone,Fl,Fun,Graph,Hecke,Higgs,loc,Pol,Quot,Rep,Spec,Supp,Std,Vect} 
\def\-{\text{-}}%
\allowdisplaybreaks

\newcommand\tabt{{\mathfrak{t}}}
\newcommand\tabs{{\mathfrak{s}}}
\newcommand\tabp{{\mathfrak{p}}}
\newcommand\tabq{{\mathfrak{q}}}
\newcommand\tabu{{\mathfrak{u}}}

\newcommand\bbla{\bblambda}

\definecolor{mediumblue}{rgb}{0.0, 0.0, 0.8}
\colorlet{darkgreen}{green!50!black}

\renewcommand*{\backref}[1]{}
\renewcommand*{\backrefalt}[4]{%
 \ifcase #1 No citations.
 \or [Page~#2.]
 \else [Pages~#2.]
 \fi%
}
\usepackage{cleveref}

\DeclareMathOperator{\ch}{char}

\newcommand{\std}{\Std}

\allowdisplaybreaks

\newtheorem{lem}{Lemma}[subsection]
\newtheorem{thm}[lem]{Theorem}
\newtheorem{cor}[lem]{Corollary}
\newtheorem{prop}[lem]{Proposition}

\newtheorem{mainthm}{Theorem}

\newtheoremstyle{lscite}{}{}{\itshape}{}{\bfseries}{}{ }{}

\theoremstyle{lscite}
\newtheorem{lemciting}[lem]{Lemma}
\newtheorem{propciting}[lem]{Proposition}
\newtheorem{thmciting}[lem]{Theorem}
\newtheorem{corciting}[lem]{Corollary}
\newtheorem{conjciting}[lem]{Conjecture}

\theoremstyle{definition}
\newtheorem{defn}[lem]{Definition}
\newtheorem{eg}[lem]{Example}

\theoremstyle{remark}

\newtheorem*{rmk}{Remarks}

\crefname{defn}{Definition}{Definitions}
\crefname{thm}{Theorem}{Theorems}
\crefname{prop}{Proposition}{Propositions}
\crefname{lem}{Lemma}{Lemmas}
\crefname{cor}{Corollary}{Corollaries}
\crefname{conj}{Conjecture}{Conjectures}
\crefname{section}{Section}{Sections}
\crefname{subsection}{Subsection}{Subsections}
\crefname{chapter}{Chapter}{Chapters}
\crefname{eg}{Example}{Examples}

\Crefname{defn}{Definition}{Definitions}
\Crefname{thm}{Theorem}{Theorems}
\Crefname{prop}{Proposition}{Propositions}
\Crefname{lem}{Lemma}{Lemmas}
\Crefname{cor}{Corollary}{Corollaries}
\Crefname{conj}{Conjecture}{Conjectures}
\Crefname{section}{Section}{Sections}
\Crefname{subsection}{Subsection}{Subsections}
\Crefname{chapter}{Chapter}{Chapters}
\Crefname{eg}{Example}{Examples}

\crefname{thmciting}{Theorem}{Theorems}
\crefname{propciting}{Proposition}{Propositions}
\crefname{lemciting}{Lemma}{Lemmas}
\crefname{corciting}{Corollary}{Corollaries}
\crefname{conjciting}{Conjecture}{Conjectures}

\Crefname{thmciting}{Theorem}{Theorems}
\Crefname{propciting}{Proposition}{Propositions}
\Crefname{lemciting}{Lemma}{Lemmas}
\Crefname{corciting}{Corollary}{Corollaries}
\Crefname{conjciting}{Conjecture}{Conjectures}

\newcommand\cala{\mathcal{A}}

\newcommand\calf{\mathcal{F}}

\newcommand\calh{\mathcal{H}}

\renewcommand\geq\geqslant
\renewcommand\leq\leqslant

\newcommand\la\lambda
\newcommand\La\Lambda
\newcommand\al\alpha
\newcommand\si\sigma
\newcommand\de\delta
\newcommand\ga\gamma

\newcommand{\ColorTableau}[2][]{
  \begin{tikzpicture}[scale=0.5, baseline=(current bounding box.center)]
    \foreach \c/\r/\col in {#1} {
       \ifx\col\empty\else
         \fill[\col] (\c-0.5, 1-\r-0.5) rectangle ++(1,1);
       \fi
    }
    \foreach \row [count=\r] in {#2} {
        \pgfmathsetmacro{\y}{1-\r} 
        \foreach \val [count=\c] in \row {
            \draw[thick, black] (\c-0.5, \y-0.5) rectangle ++(1,1);
            \node at (\c, \y) {\small \val};
        }
    }
  \end{tikzpicture}
}

\usepackage{youngtab}
\newcommand{\iyng}[1]{{\tiny\Yvcentermath1 \yng(#1)}}
\newcommand{\iyoung}[1]{{\tiny\Yvcentermath1 \young(#1)}}

\numberwithin{equation}{subsection}

\begin{document}

\title{On cellularity of Hecke Algebras for Wreath Products}

\begin{abstract}
The (generalized) Hu algebra is a nontrivial quantization of the wreath product $\Sigma_m \wr \Sigma_d$ between symmetric groups, whose representation theory controls the Hecke algebra of the complex reflection group $G(d,d,md)$. 
In this paper, we construct a unified basis for this algebra and establish its cellular algebra structure in the case $d = 2$. 
As an application, our construction provides an elementary realization of the simple modules for the Hecke algebra of type $D_{2m}$ that are parameterized by bipartitions of size $(m,m)$.
\end{abstract}

\author[Berta Hudak]{Berta Hudak}
\address{National Center for Theoretical Sciences, Taipei, Taiwan}
\email{berta.hudak@ncts.ntu.edu.tw}

\author{Chun-Ju Lai}
\address{Institute of Mathematics, Academia Sinica, Taipei 106319, Taiwan}
\email{cjlai@gate.sinica.edu.tw}



\maketitle
\section{Introduction}
\subsection{Background}
Introduced by Graham and Lehrer in \cite{GL96}, 
the notion of the cellular algebras has proved to be foundational in modular representation theory of finite dimensional algebras over a field $K$.
The framework provides a systematic construction and classification of the irreducibles via non-zero heads of cell modules $W(\la)$, generalizing the construction of irreducible modules in terms of the Specht modules $S^\la$ for the non-semisimple Hecke algebra $\cH_q(\Sigma_d)$ of the symmetric group $\Sigma_d$. 
Geck \cite{Ge07} proved that all Hecke algebras of finite Coxeter groups are cellular, provided they either have equal parameters or certain unequal parameters. 
For the complex reflection group $G(m,1,d)$, the (integral) cellular structures of the associated Hecke algebras have been studied intensively in the past decade (see \cite{Bo22} and the references therein). 

It is important to better understand the cell modules for the Hecke algebras of complex reflection group $G(r,p,n)$ as they play a crucial role in the second named author's program regarding the conjecture of Ginzburg, Guay, Opdam, and Rouquier (see Section \ref{sec:GGOR}). 
However, even in the smallest nontrivial case $G(2,2,n) = D_n$, the known construction of cell modules relies heavily on thorough understanding of Kazhdan--Lusztig cells \cite{KL79}, Lusztig's $a$-function, and asymptotic Hecke algebras \cite{Lu03}.

The approach we are taking utilizes the following Morita equivalence theorem \cite{hu02} due to Jun Hu, under a natural invertibility condition, i.e.\ $q^i \neq -1$ for $0\leq i \leq 2m-1$:
\begin{equation}\label{eq:MorD}   
\calf: 
\cH_{q}(D_{2m})\-\Mod 
\overset{\sim}{\to} 
\cala(m)\-\Mod
\oplus\prod_{1\leq i \leq m-1}\calh_q(\Sigma_i \times \Sigma_{2m-i})\-\Mod.
\end{equation}
In other words, this equivalence reduces the study of the Hecke algebra of type $D_{2m}$ to the study of a special subalgebra $\cala(m) \subseteq \cH_q(\Sigma_{2m})$, called the {\em Hu algebra}.
The odd rank case (i.e.\ $\cH_q(\Sigma_{2m-1})$) is simpler as it is similar to types B and C.  

Hu's original construction of $\cala(m)$ can be considered of type B as it relies on deep properties of the Jucys--Murphy elements in the type B Hecke algebra of unequal parameters $(1,q)$.
Recently,  a ``type $A$'' construction of the Hu algebra $\cala(m)$ was made available using the framework of the quantum wreath products  \cite{LNX}. 
\subsection{Quantum wreath products}
One should think of the Hu algebra $\cA(m)$ as a nontrivial quantization of the wreath product $\Sigma_m \wr \Sigma_2$, which is not a Coxeter group unless $m \le 2$.
The realization of $\cA(m)$ as a quantum wreath product $\cH_q(\Sigma_m) \wr \cH(2)$ (see \cref{def:QWP} for details) allows us to study its analog $\cH_q(\Sigma_m) \wr \cH(d)$ for $d> 2$, called  the {\em generalized Hu algebra}.
The representation theory of  $\cH_q(\Sigma_m) \wr \cH(d)$ is closely tied to the representation theory of the Hecke algebra for the complex reflection group $G(d,d,dm)$.

It is natural to investigate the potential cellular algebra structure of $\cH_q(\Sigma_m) \wr \cH(d)$ for an arbitrary $d$. The main difficulty arises from the key feature that each generator $H_i$ in the quantum wreath product has a complicated quadratic relation that locally takes the form $H^2 = z_m$ (see Section \ref{subs:hu} for details), where the degree zero coefficient $z_m \in \cH_q(\Sigma_m)^{\otimes 2}$ does not have a corresponding ``square root'' element in $\cH_q(\Sigma_m)^{\otimes 2}$.
Therefore, the representation theory of $\cH_q(\Sigma_m) \wr \cH(d)$ is essentially different from the representation theory of a group algebra.

\subsection{A new cellular basis}

Consider the index poset $\Omega$ which parametrizes the Specht modules $S^\bbla\in \cH_q(\Sigma_m) \wr \cH(d)$-Mod constructed in \cite{LNXII}.
Each element $\bbla \in \Omega$ is a multiparititons of $d$ with support being a subset of the set $\Pi_m$ of partitions of $m$.
The  tableaux set $M(\bbla)$ (see \cref{defn:m-set}) consists of triples $\bbT = (w_\bbT, \tabs_\bbT, \tabt_\bbT)$ consisting of (i) a certain shortest coset representative $w_{\mathbb{T}}$,
    (ii) a certain tuple of standard tableaux $\mathfrak{s}_{\mathbb{T}} = (\tabs_\bbT^1,\tabs_\bbT^2,\dots)$ of total size $d$, and
    (iii) a certain tuple of standard tableaux $\mathfrak{t}_{\mathbb{T}} = (\tabt_\bbT^1, \dots, \tabt_\bbT^d)$ of size $(m, m, \dots, m)$.
We construct a new cellular basis element $C_{\bbT,\bbT'}^{\bbla}$ (see \cref{def:CTT}) using a variant of the full Young symmetrizer $z_{\nu}^{\lambda}$ sitting inside the generic Hecke algebra $\mathcal{H}_{(0,f_{\lambda})}(\Sigma_{d})$. 
This basis does not specialize to existing cellular bases associated to wreath products in the literature (for example, \cite{GG13, G20, T24}) due to the interplay with $z_m$.
\begin{mainthm}[{\cref{thm:cellularity}}]\label{thm:A}
The Hu algebra $\cA(m) \cong \cH_q(\Sigma_m) \wr \cH(2)$ is a cellular algebra.
\end{mainthm}
Our cellular basis recovers established bases in the extremal cases. When $d=1$, the basis element $C_{\bbT,\bbT'}^{\bbla}$  specializes exactly to the Murphy basis element for the Hecke algebra $\mathcal{H}_q(\Sigma_m)$.   
When $m=1$, the generalized Hu algebra becomes the generic Hecke algebra $\mathcal{H}_{(0,(1+q)^2)}(\Sigma_d)$. 
In this scenario, the element $C_{\bbT,\bbT'}^{\bbla}$ coincides with the full symmetrizer basis element for this algebra.  

\subsection{Cell modules versus Specht modules}
The cell modules constructed by Geck are intrinsically tied to geometry. 
However, their structures remain notoriously elusive for general Hecke algebras --
with one exception of type A, where these modules admit a realization \cite{MP05} via the tableaux-based approach of Dipper, James, and Murphy \cite{djm,mur92}.

Such an identification remains only partially understood for type $B$ (with unequal parameters) and is unknown beyond types $A$ and $B$.
For example, the complex reflection group  $G(m,1,d)$ has a quantization known as the Ariki-Koike algebra \cite{AK94} (or the cyclotomic Hecke algebra \cite{djm}) denoted by
$\cH_{\mathbf{q}}(C_m\wr \Sigma_d)$. 
Under certain assumptions on the parameters $\mathbf{q}$, one can combine results on rational Cherednik algebras by Shan \cite[Lemma 3.1]{Sh11} and by Chlouveraki, Gordan, and Griffeth \cite[Proposition 4.6]{cgg12} to deduce that $[W(\lambda)] = [S^\lambda] \in K(\cH_{\mathbf{q}}(C_m\wr \Sigma_d))$.
In other words, the best we know is that at the Grothendieck group level, Geck's cell modules agree with the Specht modules over $\cH_{\mathbf{q}}(C_m\wr \Sigma_d)$.

Our cellularity result extends such a module identification to the case of the Hu algebras (and hence of type D).
\begin{mainthm}[\cref{thm:CellIsSpecht}]\label{thm:B}
For $\bbla \in \Omega$, the cell module $W(\bbla)$ is isomorphic to the Specht module $ S^\bbla$ over the Hu algebra $\cA(m)$.
\end{mainthm}
Furthermore, we compute the canonical bilinear form $\phi_\bbla$ for all $\bbla \in \Omega$ and then obtain representation theoretic consequences in \cref{thm:cellinfo}. 
\subsection{}
The paper is organized as follows. 
In Section 2, we review the necessary background on cellular algebras, the Hu algebra, and its generalization via the  quantum wreath products.

In Section 3, we construct the combinatorial data required for our main results for an arbitrary $d \ge 2$. This includes the poset consisting of poset-indexed multipartitions and the associated multitableaux sets. 
We then explicitly construct our new cellular basis elements, and verify that this basis successfully recovers established cellular bases in the extremal cases $d=1$ and $m=1$.

Section 4 is devoted to the proof of Theorem A. We provide explicit formulas for the basis elements and rigorously verify the cellular algebra axioms. We also briefly discuss how the resulting cellular datum naturally yields a generalized Robinson-Schensted correspondence.

In Section 5, we analyze the structure of the resulting cell modules to prove Theorem B. We demonstrate that these cell modules are isomorphic to the Specht modules constructed by Hu (using the language of wreath modules), and we explicitly compute their canonical bilinear forms.

Finally, Section 6 explores the broader implications of our work. We detail how our cellularity framework yields a direct realization of the simple modules for the Hecke algebra of type $D_{2m}$, and we discuss anticipated connection with rational Cherednik algebras.
\section{Background}\label{sec:backgr}
\subsection{Cellular algebras}
In this section, we recall the notion of cellular algebras \cite{GL96} over a commutative ring $K$.
\begin{defn}
    A {\em cellular algebra} is an associative $K$-algebra $A$ with cell datum $(\Omega,M,C,*)$ where
    \begin{itemize}
        \item[C1)] $\Omega$ is a partially ordered set and for each $\la\in \Omega$, $M(\la)$ is a finite set such that the map
        \begin{equation*}
            C: \bigsqcup\nolimits_{\la\in\Omega}M(\la)^2 \rightarrow A
        \end{equation*}
        is injective. Moreover,  $\{C^\la_{\tabt, \tabt'}:=C(\tabt, \tabt') ~|~ (\tabt, \tabt')\in M(\la)^2, \ \la\in\Omega\}$ forms a $K$-basis of $A$.
        \item[C2)] $*$ is an anti-involution of $A$ such that $(C^{\la}_{\tabt,\tabt'})^*=C^{\la}_{\tabt',\tabt}$ for all $\tabt, \tabt' \in M(\la)$.
        \item[C3)] For any element $a\in A$, there exists a map $r_a:\bigsqcup_{\la\in\Omega} M(\la)^2 \to K$ such that
        \begin{equation*}
            a C^{\la}_{\tabt,\tabt'} \equiv \sum_{\tabs\in M(\la)} r_a(\tabs,\tabt)C^{\la}_{\tabs,\tabt'} 
            \pmod{A_{<\la}},
            \quad
            \textup{for all}
            \quad\tabt, \tabt'\in M(\la),
        \end{equation*}
        where $A_{<\la}$ denotes the $K$-submodule of $A$ generated by $\{C^{\mu}_{\tabt, \tabt'}\mid \mu<\la,\ \tabt, \tabt'\in M(\mu)\}$.
    \end{itemize}
\end{defn}
For each $\la\in \Omega$, there exists a {\em cell module} $W(\la)$ which is the (left) $A$-module such that it is free as an $K$-module with basis $\{C^\la_{\tabt}\mid \tabt\in M(\la)\}$ and $A$-action given by, for $a \in A$:
\begin{equation}
    aC_\tabt^\la=\sum\nolimits_{\tabs\in M(\la)}r_a(\tabs,\tabt)C_{\tabt}^\la.
\end{equation}
Let  $\phi_\lambda: W(\lambda) \times W(\lambda) \to K$  be the {\em canonical bilinear form} constructed via
\begin{equation}    
\label{eq:phi_lambda}
C^\lambda_{\tabs_1,\tabs_2} C^\lambda_{\tabt_1,\tabt_2} \equiv \phi_\lambda(C_{\tabs_2}, C_{\tabt_1}) C^\lambda_{\tabs_1,\tabt_2} \mod A_{<\lambda}.
\end{equation}
Let $\rad(\la):= \{x \in W(\la)~|~ \phi_\la(x,y)=0 \textup{ for all }y \in W(\la)\}$.
The above cellular data provides an insightful information to the representation theory of $A$ as follows.
\begin{prop}[{\cite[Theorem 3.4, Theorem 3.8, Remark 3.10]{GL96}}] \label{prop:philainfo}
Suppose that $K$ is a field and that $A$ is a cellular algebra with cell datum $(\Omega,M,C,*)$. Then,
\begin{enua}
    \item Let $\Omega_0 := \{\la \in \Omega ~|~ \phi_\la \neq 0\}$.
    The set $\{ W(\la)/\rad(\la)~|~ \la \in \Omega_0\}$ forms a complete set of  absolutely irreducible $A$-modules, up to isomorphisms.
    \item Semisimplicity of $A$ is equivalent to that $\rad(\la) = 0$ for all $\la \in \Omega$.
    \item If $\phi_\la \neq 0$ for all $\la \in \Omega$, then $A$ is quasi-hereditary.
\end{enua}
\end{prop}
\subsection{Tableaux Combinatorics}
By $\mu \vDash m$ we mean $\mu$ is a \emph{composition} of $m$.
That is, a sequence $\mu = (\mu_1, \mu_2, \dots)$ of positive integers with $\sum_i \mu_i = m$.  
Denote by $\Lambda_m$ the set of compositions of $m$.
Let $\Sigma_m$ be the symmetric group on $m$ letters, and let $s_i = (i~i+1)$ be the simple transpositions.
Any composition $\mu \vDash m$ has a corresponding Young subgroup $\Sigma_\mu$ generated by those $s_i$ such that $i \not\in \{\mu_1, \mu_1+\mu_2, \dots\}$.

By $\lambda \vdash m$ we mean $\la$ is \emph{partition} of $m$, that is, a composition with $\lambda_1 \ge \lambda_2 \ge \cdots$.
Let $\Pi_m$ be the set of partitions of $m$.
We identify $\lambda$ with its Young diagram $[\lambda] := \{(i,j) \in \mathbb{Z}_{>0}^2 ~|~ 1\leq j \leq \lambda_i\}$.
The English notation for $[\la]$ is used in this paper, e.g.
\[
\la = (5,3) \vdash 8,
\quad
[\la] = \iyng{5,3}.
\]
A \emph{tableau} of shape $\lambda \vdash m$ is a filling of the diagram $[\lambda]$ with entries in $\mathbb{Z}_{>0}$.
Let $\la^t$ be the transposed partition of $\la$. 
A tableau $\tabt$ is called \emph{standard} if it contains the numbers $1,\dots,m$ each exactly once, with rows strictly increasing from left to right and columns strictly increasing from top to bottom.
Denote by $\std(\lambda)$ the set of standard tableaux of shape $\lambda$.

The set $\std(\la)$ affords a right $\Sigma_m$-action by permuting the entries of a tableau. 
Let $\tabs_R^\la \in \std(\la)$ 
be the initial tableau where the numbers $1, 2, \dots, d$ are filled in increasing order along the rows, and let $\tabs_C^\la := (\tabs_R^{\la^t})^t$ be the other initial tableau where the numbers $1, 2, \dots, d$ are filled in increasing order along the columns.
Let $d_\tabt, w(\la) \in \Sigma_m$ be the shortest element such that $\tabs_R^\la \cdot d_\tabt = \tabt$ and $\tabs_R^\la \cdot w(\la) = \tabs_C^\la \cdot$.
\begin{eg}
    Let $\la= (2,1) \vdash 3$. Then $\std(\la) =\{ \tabs_C^\la = \iyoung{12,3}, \iyoung{13,2}\}$.
    
    For $\tabt = \iyoung{13,2}$, one has $\tabs_R^\la \cdot (2 3) = \tabt = \tabs_C^\la$ and hence $d_\tabt = s_2 = w(\la)$.
\end{eg}
\subsection{Hecke algebras}
Fix a field $K$ of characteristic $\ch K \neq 2$.
For elements $a, b \in K$, denote by the {\it generic Hecke algebra} for  the symmetric groups $\Sigma_m$ on $m$ letters by the $K$-algebra $\cH_{(a,b)}(\Sigma_m)$ generated by
$T_1, \dots, T_{m-1}$ subject to the following relations:
\begin{align*}
&T_i T_j T_i= T_j T_i T_j,
\quad T_i T_l = T_l T_i,
&\textup{if }|i-j| = 1,\ |i-l| > 1,
\\
&T_i^2 = aT_i +b, 
&\textup{if }1\leq i \leq m-1.
\end{align*}
For any invertible $\phi \in K$, the generic Hecke algebra $\cH_{(0,\phi^2)}(\Sigma_m)$ is isomorphic to the group algebra $K[\Sigma_m]$ by the assignment $T_i \mapsto \phi s_i$ where $s_i := (i ~ i+1)$.
Let $\cH_q(\Sigma_m) = \cH_{(q-1,q)}(\Sigma_m)$ for any invertible $q \in K$.
The algebras $\cH_{(0,\phi^2)}(\Sigma_m)$ and $\cH_q(\Sigma_m)$ are both cellular with respect to the following cell data:  
\[
\Omega = \Pi_m,
\quad
M = \std(\la),
\quad
C^\la_{\tabt, \tabt'} = m^\la_{\tabt, \tabt'},
\quad
*:T_i \mapsto T_i,
\]
where $m^\la_{\tabt, \tabt'} := T_{d_\tabt}^* x_\la T_{d_{\tabt'}}$ is the Murphy basis and 
\begin{equation}
x_\la := \begin{cases}
    \sum_{w \in \Sigma_\la} T_w, &\textup{if } (a,b) = (q-1,q);
    \\
    \sum_{w \in \Sigma_\la} \phi^{-\ell(w)}T_w, &\textup{if } (a,b) = (0,\phi^2),
\end{cases}
\end{equation}
where $\ell$ is the length function.
The cell module $W(\la)$ is known \cite{GL96} to be isomorphic to the Specht module $S^\la$ for the Hecke algebra.
Moreover, the Robinson--Schensted correspondence 
$\Sigma_m \to \bigcup_{\la\in \Pi_m} \std(\la)^2$, $w \mapsto (\tabp(w),\tabq(w))$
has a tight connection with the cellular bases as $\cH_q(\Sigma_m)$ has another (geometric) cellular datum in which the cellular basis coincide with the Kazhdan--Lusztig basis $\{\underline{H}_w\}_w$. To be precise, $C^\la_{\tabp(w), \tabq(w)} := \underline{H}_w$ (see \cite[Theorem 5.6.2]{W03}).
\subsection{Hu algebras}\label{subs:hu}

From now on, we assume that $K$ is a field with $\ch K \neq 2$ and that $q \in K^\times$.
We will use the identification $\calh_q(\Sigma_m \times \Sigma_m)
\cong \calh_q(\Sigma_m) \otimes \calh_q(\Sigma_m) $ via the following canonical isomorphism:
\begin{equation*}
    T_i\mapsto\begin{cases}
        T_i\otimes 1\hspace{1.1cm}\text{if }1\le i \le m-1,\\
        1\otimes T_{i-m}\hspace{.65cm}\text{if }m+1\le i\le2m-1.
    \end{cases}
\end{equation*}
\begin{defn}
    The Hu algebra is the subalgebra 
    \[
    \cala(m) := \langle
    \calh_q(\Sigma_m)^{\otimes 2}, h_m^*
    \rangle \subseteq \calh_q(\Sigma_{2m})
    \]    
    where $h_m$  is the unique element in $\calh_q(\Sigma_{2m})$ such that the following holds as elements in the Hecke algebra of type B of unequal parameters $(1,q)$:
    \begin{equation}
    u_m^- h_m u_m^+ \in h_m u_m^+ + \sum_{j>m}     \calh_q(\Sigma_{2m}) u_j^+ \calh_q(\Sigma_{2m}).
    \end{equation}
    Here, $u_1^{\pm1} := 1$ and $u_{i+1}^\pm := u_i^{\pm1} (1\pm T_i \dots T_1 T_0 T_1 \dots T_i)$.
\end{defn}
A precise formula for $h_m$ was obtained in \cite{LNX}, but we won't need it in this paper. Instead, we collect some favorable properties of $h_m$ as follows.




\begin{prop}[{\cite[(4.12)]{dj92}, \cite[Lemmas 1.8, 1.10, 3.2, Theorem 3.5]{hu02}}]
Suppose that $f_{2m}(q):=	\prod_{i=0}^{2m-1}(1+q^i) \neq 0$. Then
    \begin{enua}
        \item The elements $h_m^*$ and $z_m:= (h_m^*)^2 \in \cH_q(\Sigma_{2m})$ are invertible in $\cH_q(\Sigma_m)^{\otimes 2}$. 
        \item The element $h_m^*$ intertwines $\cH_q(\Sigma_{m})\otimes \cH_q(\Sigma_{m})$, i.e., 
        \[
        h_m^*(T_x \otimes T_y) = (T_y \otimes T_x) h_m^*
        \quad
        \textup{for all}
        \quad
        x,y \in \Sigma_m.
        \]
        \item The element $z_m$ lies in the center $Z(\cH_q(\Sigma_{m})\otimes \cH_q(\Sigma_{m}))$, and it acts on the Specht module $S^{\la'}\otimes S^{\la''} \in \cH_q(\Sigma_m)\otimes \cH_q(\Sigma_m)$-mod by a scalar $f_{\la',\la''} \in K^\times$.
        \item There exists an element $\sqrt{f_\la} \in K$ such that $\sqrt{f_\la}^2 = f_\la := f_{\la,\la}$.
    \end{enua}
\end{prop}
\begin{eg}
When $m=1$, it is understood that $\cH_q(\Sigma_1) = K$, $h_1 = T_1 + qT_1^{-1}$ and 
\begin{equation}   \label{eq:z1}
z_1 = (T_1 + qT_1^{-1})^2 = (T_1 - qT_1^{-1})^2 + 4q = (q-1)^2 + 4q = (q+1)^2 \in K.
\end{equation}
For $m=2$, it is not easy to describe the element $z_2$ in terms of $T_1$ and $T_3$.
That being said, one can compute explicitly the following scalars for the $z_2$-action on the Specht modules:
\begin{equation}   \label{eq:z2}
f_{(1,1)} = f_{(2)} = (1+q)^2 (1+q^2)^2,
\quad
f_{(1,1),(2)} = f_{(2),(1,1)} = 4(1+q) (1+q^3).
\end{equation}
We set $\sqrt{f_{(1,1)}} := (1+q)(1+q^2) =: \sqrt{f_{(2)}}$.
\end{eg}
\begin{cor}\label{cor:zmaction}
Suppose that $\la, \la' \vdash m$, $\tabs_1, \tabs_2 \in \std(\la)$ and $\tabt_1, \tabt_2\in \std(\la')$. Then,
$z_m (C_{\tabs_1, \tabs_2}^\la \otimes C_{\tabt_1, \tabt_2}^{\la'}) = f_{\la,\la'} (C_{\tabs_1, \tabs_2}^\la \otimes C_{\tabt_1, \tabt_2}^{\la'}).$
\end{cor}
\begin{proof}
    Let $a_i, a'_i \in \cH_a(\Sigma_m)$ be such that $z_m = \sum_i a_i \otimes a'_i$.
    Since 
    \[
    z_m (C_{\tabs_1}^\la \otimes C_{\tabt_1}^{\la'}) 
    = \sum_i 
    \sum_{\tabu\in\std(\la)}
    \sum_{\tabu'\in\std(\la')}
    r_{a_i}(\tabu, \tabs_1) 
    r_{a'_i}(\tabu', \tabt_1)
    (C_{\tabs_1}^\la  \otimes  C_{\tabt_1}^{\la'}  ),
    \]
one has  
$   \sum_i\sum_{\tabu\in\std(\la)}
    \sum_{\tabu'\in\std(\la')}
    r_{a_i}(\tabu, \tabs_1) 
    r_{a'_i}(\tabu', \tabt_1)
 = f_{\la,\la'}$, and hence
\[
\begin{split}
z_m (C_{\tabs_1, \tabs_2}^\la \otimes C_{\tabt_1, \tabt_2}^{\la'}) 
&=
    \sum_i
   \sum_{\tabu\in\std(\la)}
    \sum_{\tabu'\in\std(\la')}
    r_{a_i}(\tabu, \tabs_1) 
    r_{a'_i}(\tabu', \tabt_1)
    (C_{\tabs_1, \tabs_2}^\la \otimes C_{\tabt_1, \tabt_2}^{\la'}) 
\\
&= f_{\la,\la'} (C_{\tabs_1, \tabs_2}^\la \otimes C_{\tabt_1, \tabt_2}^{\la'}).
\qedhere
\end{split}
\]
\end{proof}
\subsection{Quantum wreath products}
Hu's original definition works for the case $d=2$.
An analog for $d >2$ can be constructed using the notion of the quantum wreath products.
Let $B$ be an associative $K$-algebra, and let $d \ge 2$ be an integer.
For $ 1\leq i \leq d-1$, write
\begin{align*}
    &X_i := 1^{\otimes i-1} \otimes X \otimes 1^{\otimes d-i-1},
    &
    (X \in B\otimes B),
\\
    &\psi_i(b_1\otimes \dots \otimes b_d) := \psi(b_i\otimes b_{i+1})_i,
    &
    (\psi \in \End_K(B\otimes B), b_j \in B).
\end{align*}
Let $Q=(R,S,\rho,\sigma)$ be a quadruple where $R,S\in B\otimes B$, $\rho,\sigma\in \End_K(B\otimes B)$.
Thus, $S_i$, $R_i$, $\sigma_i(b)$, and $\rho_i(b)$ are all elements in $B^{\otimes d}$, for $b\in B^{\otimes d}$ and $1\leq i \leq d-1$.
\begin{defn}\label{def:QWP}
The \emph{quantum wreath product} is the associative $K$-algebra generated by the algebra $B^{\otimes d}$ and elements $H_1,\dots,H_{d-1}$ such that for $1\le k\le d-2$, $1\le i\le d-1$, and $|j-i|\ge 2$ we have that :
\begin{align}
\text{(braid relations)} \qquad
&H_kH_{k+1}H_k = H_{k+1}H_kH_{k+1}, \qquad
H_iH_j = H_jH_i, \label{eq:braid} \\
\text{(quadratic relations)} \qquad
&H_i^2 = S_iH_i + R_i, \\
\text{(wreath relations)} \qquad
&H_ib = \sigma_i(b)H_i + \rho_i(b)
\qquad (b\in B^{\otimes d}). 
\end{align}
\end{defn}

We refer to this algebra as $B \wr_Q H(d)$, or $B \wr H(d)$ whenever it is convenient.

\begin{defn}
Consider the quantum wreath product $B \wr \calh(d)$ with the following choices:
\[
B = \calh_q(\Sigma_m), \qquad
S = 0, \qquad
R = z_m, \qquad
\sigma : a \otimes b \mapsto b \otimes a, \qquad
\rho = 0.
\]
We call these quantum wreath products the {\it generalized Hu algebra} since $\cala(m) \cong B \wr \cH(2)$ via $h_m^* \mapsto H_1, \cH_q(\Sigma_m\times \Sigma_m) \mapsto \cH_q(\Sigma_m)^{\otimes 2}$.
\end{defn}

\begin{prop}[{\cite[Corollary 7.3.2]{LNXII}}] \label{gHuast}
    The following is an anti-automorphism of the generalized Hu algebra $\cH_q(\Sigma_m) \wr \cH(d)$:
\[
\ast: B \wr \cH(d) \to B \wr \cH(d), \quad 
b_1 \otimes \dots \otimes b_d \mapsto b_1^{\ast} \otimes \dots \otimes b_d^{\ast}, \quad
H_i \mapsto H_i.
\]
\end{prop}

\section{Construction of Cell Data}\label{sec:celldata}
\subsection{The Index Poset $\Omega$}\label{sec:Omega}
For the rest of the paper, we set $B=\calh_q(\Sigma_m)$.
Let $(\Pi_{m}^d, \le)$ be the poset of multipartitions of $(m,m, \dots, m)$ of $d$ components with the usual dominance order, and let $\Pi = \bigcup_{n\geq 0} \Pi_n$ be the set of all (possibly empty) partitions.
\begin{defn}\label{defn:lambda}
For any set $I$, let $\Pi_I(d)$ be the $I$-indexed multipartitions of $d$, i.e.
\[
\Pi_I(d) = \left\{
\bbla: I \to \Pi 
~\middle|~
\sum\nolimits_{\la\in I} \#\bbla(\la) = d
\right\}.
\]
An arbitrary $\bbla \in \Pi_I(d)$ can be written as a sum of characteristic functions of the form $e_\la^\nu: I \to \Pi$ that is supported on $\{\la\}\subseteq I$ with its value given by $e_\la^\nu(\la) = \nu \in \Pi$. 
Let $\Omega := \Pi_{\Pi_m}(2)$, and write $[\la', \la''] := e_{\la'}^{\iyng{1}}+e_{\la''}^{\iyng{1}}$ for short. 
Observe that $[\la', \la''] = [\la'', \la']$.
\end{defn}
\begin{defn}[{Dominance Order \cite[Definition 6.1.1]{LNXII}}]
\label{def:domPO}
Suppose that $(I,\leq)$ is a poset.
Define the dominance order on $\Pi_{\Pi_m}(d)$ by $\bblambda \leq \bblambda'$ if and only if 
\[
\sum_{j \leq k}\bblambda(\nu)_j
+
\sum_{\gamma < \nu \in I} \#\bblambda(\gamma)  
\ \leq \ 
\sum_{j \leq k}\bblambda'(\nu)_j
+
\sum_{\gamma < \nu \in I} \#\bblambda'(\gamma)
\qquad
(\textup{for all}
~ \nu\in I, 
~k\geq 1).
\]
\end{defn}
\begin{eg}
Suppose that $m=d=2$. Take the dominance order on $I = \Pi_2$ so $(1,1) = \iyng{1,1} < (2) = \iyng{2}$.
The dominance on $\Omega$ is given by
\[
 e_{(1,1)}^{(1,1)} < e_{(1,1)}^{(2)} < [(1,1), (2)] < e_{(2)}^{(1,1)} < e_{(2)}^{(2)}.
\]    
\end{eg}
Fix an enumeration\footnote{
The standard convention is $\{\pi^{(1)}, \pi^{(2)}, \dots \}$. In this paper, we emphasize that the superscripts for partitions serve purely as indexing decorations and do not denote powers or multiplicities. This deliberate departure from standard notation is to avoid ambiguity, as we frequently use superscripts of the form $(i)$ to denote the one-row partitions consisting of $i$ box(es).} $\{\pi^{1}, \pi^2, \dots \}$ of $\Pi_m$ that is compatible with the partial order $\leq$, i.e.,  $\pi^{i}<\pi^{j}$ in $\Pi_m$ implies $i<j \in \mathbb{N}$.
Given $\bblambda:I_B\to \Pi$ with support supp$(\bblambda) = \{\pi^{k_1}, \dots, \pi^{k_r}\}$ where $1\leq k_1 < \dots < k_r \leq s$,
define a composition $\mu = \mu(\bblambda) \vDash d$ by $\mu_i := \#\bblambda(\pi^{k_i})$, and a multipartition
\begin{equation}    
\overline{\bblambda} := (
\underset{\mu_1\textup{ times}}{\underbrace{\pi^{k_1}, \dots, \pi^{k_1}}}, 
\underset{\mu_2\textup{ times}}{\underbrace{\pi^{k_2}, \dots, \pi^{k_2}}}, 
\dots, 
\underset{\mu_r\textup{ times}}{\underbrace{\pi^{k_r}, \dots, \pi^{k_r}}} 
) \in \Pi_m^d.
\end{equation}
Equivalently, $\overline{\bblambda}$ is a map $\{1,\dots,d\} \to \Pi$ such that $\overline{\bblambda}(i) := \pi^{k_j}$ where $k_1 + \dots + k_{j-1} < i \leq k_1 + \dots + k_j$.
\begin{eg}\label{eg:d=2}
\begin{enua}
\item 
Suppose that $m=3$, $d=4$, $\Pi_3 = \{(1,1,1) < (2,1) < (3)\}$. 
Let $\bbla = e_{(3)}^{(1)}+e_{(1,1,1)}^{(2,1)}$. Then,
$\overline{\bbla} = ((1,1,1), (1,1,1), (1,1,1), (3)) \in \Pi_3^4$.
\item Suppose that $d=2$. Elements $\bbla \in \Omega$ are of the following form:
$e_\la^{(2)}$, $e_\la^{(1,1)}$, and $[\pi^{i}, \pi^{j}]$ for $i< j$.
The corresponding $\overline{\bbla}$ are, respectively, $(\la, \la)$, $(\la, \la)$, and $(\pi^{i}, \pi^{j})$.
\end{enua}
\end{eg}

\subsection{The Tableaux Sets $M(\bbla)$}

Let $\mu \vDash d$ be a composition.
Denote the corresponding set of shortest representatives for the right cosets $\Sigma_\mu \backslash \Sigma_d$ by
\[
{}^\mu\Sigma := \{w \in \Sigma_d ~|~  g w > w \textup{ for all } g\in \Sigma_\mu\}.
\]
\begin{defn}\label{defn:m-set}
For $\bbla\in \Omega$, define
\[
\begin{split}
M(\bbla) &:= 
{}^{\mu(\overline{\bbla})}\Sigma
\times
\prod\nolimits_{j \geq 1}\std(\bbla(\pi^{j}))
\times
\prod\nolimits_{i=1}^d
\std(\overline{\bbla}(i))
\\
&\cong 
\std(\bbla(\pi^{1}),\bbla(\pi^{2}), \dots)
\times
\prod\nolimits_{i=1}^d
\std(\overline{\bbla}(i))
.
\end{split}
\]
In words, each $\bbT \in M(\bbla)$ is a triple consisting of 
(i) a shortest coset representative $w_\bbT \in {}^{\mu(\overline{\bbla})}\Sigma$,
(ii) a tuple of standard tableaux $\tabs_\bbT = (\tabs_\bbT^{1}, \tabs_\bbT^{2}, \dots)$ of total size $d$, and 
(iii) a tuple of standard tableaux $\tabt_\bbT = (\tabt_\bbT^{1}, \dots, \tabt_\bbT^{d})$ of size $(m,m, \dots, m)$.

Note that we distinguish the two sets $\std(\bbla(\pi^{1}), \bbla(\pi^{2}),\dots)$ 
and $\prod_{j\geq 1} \std(\bbla(\pi^{j}))$.
The former contains multitableaux whose boxes are filled with the numbers from 1 to $d$.
In the latter elements are tuples in which the boxes in the $j$th component are filled with numbers from 1 to $\#\bbla(\pi^{j})$.

We could have used $M(\bbla) = \std(\bbla(\pi^{1}),\bbla(\pi^{2}), \dots)
\times
\prod\nolimits_{i=1}^d
\std(\overline{\bbla}(i))
$ as the definition of the tableaux sets, however for our construction keeping (i) and (ii) separate makes the formulas neater.
\end{defn}
\begin{eg}\label{eg:d=2'}
\begin{enua}    
\item
Let $m=3$, $d=4$ as in \cref{eg:d=2}(a).
Any $\bbT \in M(\bbla)$ is a triple
$(w_\bbT, \tabs_\bbT, \tabt_\bbT)$ where
\[
w_\bbT \in \{e, s_3, s_3s_2, s_3s_2s_1\},
\quad
\tabs_\bbT \in \left\{
\left(
{\tiny\Yvcentermath1\young(12,3)}, 
{\tiny\Yvcentermath1\young(1)} 
\right),
\left(
{\tiny\Yvcentermath1\young(13,2)}, 
{\tiny\Yvcentermath1\young(1)} 
\right)
\right\},
\quad
\tabt_{\bbT} = 
\left(
{\tiny\Yvcentermath1\young(1,2,3)}, 
{\tiny\Yvcentermath1\young(1,2,3)}, 
{\tiny\Yvcentermath1\young(1,2,3)}, 
{\tiny\Yvcentermath1\young(123)} 
\right).
\]
The pairs $(w_\bbT, \tabs_\bbT)$ are in bijection with the following bitableaux of size $(3,1)$:
\[
\begin{array}{c|ccccccc}
     w_\bbT& e & s_3 & s_3 s_2 & s_3s_2s_1   
     \\
     \hline
     \tabs_\bbT = \left(
{\tiny\Yvcentermath1\young(12,3)}, 
{\tiny\Yvcentermath1\young(1)} 
\right)
    &
        \left(
{\tiny\Yvcentermath1\young(12,3)}, 
{\tiny\Yvcentermath1\young(4)} 
\right)
    &
        \left(
{\tiny\Yvcentermath1\young(12,4)}, 
{\tiny\Yvcentermath1\young(3)} 
\right)
    &     
         \left(
{\tiny\Yvcentermath1\young(13,4)}, 
{\tiny\Yvcentermath1\young(2)} 
\right)
    &
         \left(
{\tiny\Yvcentermath1\young(23,4)}, 
{\tiny\Yvcentermath1\young(1)} 
\right)
     \\
     \tabs_\bbT =
     \left(
{\tiny\Yvcentermath1\young(13,2)}, 
{\tiny\Yvcentermath1\young(1)} 
\right)
    &
         \left(
{\tiny\Yvcentermath1\young(13,2)}, 
{\tiny\Yvcentermath1\young(4)} 
\right)
    &
         \left(
{\tiny\Yvcentermath1\young(14,2)}, 
{\tiny\Yvcentermath1\young(3)} 
\right)
    &     
        \left(
{\tiny\Yvcentermath1\young(14,3)}, 
{\tiny\Yvcentermath1\young(2)} 
\right)
    &
         \left(
{\tiny\Yvcentermath1\young(24,3)}, 
{\tiny\Yvcentermath1\young(1)} 
\right)
\end{array}
\]
\item
Let $d=2$ as in \cref{eg:d=2}(b). 
Recall that $\tabs^\nu_R \in \std(\nu)$ is the row-initial tableau.
One has, for $\nu \vdash 2$:
\[
M(e_\la^{\nu}) =
\{ (e, \tabs^\nu_R, (\tabt_\bbT^{1},\tabt_\bbT^{2}))
\in \Sigma_1\times \std(\nu) \times \std(\la)^2\}.
\]
That is, $M(e_\la^{\nu}) \cong \std(\la)^2$.
Next, for distinct partitions $\pi^{i}$ and $\pi^{j}$ with $i < j$:
\[
M([\pi^{i},\pi^{j}]) =
\{(w_\bbT, (\tabs_R^{\iyng{1}}, \tabs_R^{\iyng{1}}), (\tabt_\bbT^{1},\tabt_\bbT^{2})) 
\in \Sigma_2 \times \std(\iyng{1})^2\times (\std(\pi^{i}) \times \std(\pi^{j}))
\}.
\]
 In other words, $M([\la,\la']) \cong \Sigma_2 \times \std(\la) \times \std(\la')$ where $\la \neq \la' \vdash m$.
\end{enua}
\end{eg}
\subsection{The Cellular Basis}
Motivated by the construction of the Specht modules for the twisted Hecke algebra as in \cite[Section 3.3]{LNXII}, we define the following variant of the Young symmetrizer. 
Consider the generic Hecke algebra $\mathcal{H}_{(0, f_\la)}(\Sigma_d)$ where $\la \vdash m$.
For $\nu \vdash d$, $\tabt \in \std(\nu)$, denote by
$z^\la_\nu \in \mathcal{H}_{(0, f_\la)}(\Sigma_d)$ the Young symmetrizer 
\begin{equation}    
z^\la_\nu := x^\la_\nu T_{w(\nu)}y_{\nu^t}^\la ,
\quad
\textup{where}
\quad
x_\nu^\la = \sum_{w \in \Sigma_\nu} (-\sqrt{f_\la})^{\ell(w)} T_w,
\quad
y_\nu^\la := \sum_{w \in \Sigma_\nu} (-\sqrt{f_\la})^{-\ell(w)} T_w.
\end{equation}
Hence, $z^\la_\nu \mathcal{H}_{(0, f_\la)}(\Sigma_d)$ is isomorphic to the Specht module for the twisted Hecke algebra in \cite{LNXII} (referred to as $S^\nu_{\chi_\lambda} \in \cH^{\chi_\la}_d$-mod therein).
Moreover, let $\zeta^\la(x) \in K$ be such that $z_{\nu}^{\la} = \sum_{x\in \Sigma_d} \zeta^\la(x) T_x$.
\begin{eg}
    For $n\le 2$, one has 
\begin{equation}\label{eq:zeta}
z^\la_{\iyng{1}} = 1,
\quad
z^\la_{\iyng{2}} = 1+\sqrt{f_\la}^{-1} T_1,
\quad
z^\la_{\iyng{1,1}} = 1 -\sqrt{f_\la}^{-1} T_1.
\end{equation}
\end{eg}
\begin{defn}\label{def:CTT}
    Let $\bbT, \bbT' \in M(\bbla)$. Define
\begin{equation}\label{def:CST}   
C^\bbla_{\bbT, \bbT'} := \sum_{x  \in \Sigma_{d}} 
\zeta_{\bbT, \bbT'}(x)
H_{w_\bbT^{-1}}
\Big(
C_{\tabt_\bbT^{1}, \tabt_{\bbT'}^{x(1)}}^{\overline{\bbla}(1)}
\otimes
\dots
\otimes
C_{\tabt_\bbT^{d}, \tabt_{\bbT'}^{x(d)}}^{\overline{\bbla}(d)}
\Big)
H_{d_{\tabs_\bbT}^{-1}}
H_{x}
H_{d_{\tabs_{\bbT'}}}
H_{w_{\bbT'}},
\end{equation}
where 
$
\zeta_{\bbT, \bbT'}(x)
:=
\prod_{j\geq 1}
\zeta^{\pi^j}_{\bbla(\pi^j)}(x).
$
\end{defn}
\subsection{Extremal Cases}
In this section, we show that our cellular basis element recovers the Murphy basis for the extremal cases $\Sigma_1 \wr \Sigma_d$ and $\Sigma_m \wr \Sigma_1$ in two distinct flavors.

\begin{lem}\label{lem:md=1}
\begin{enua}
\item
Suppose that $d=1$.
The element $C^\bbla_{\bbT, \bbT'}$ in the Hecke algebra $\cH_q(\Sigma_m)$ is given by 
\[
C^\bbla_{\bbT,\bbT'} = C^\la_{\tabt_\bbT, \tabt_{\bbT'}},
\]
i.e.\ the usual Murphy basis element for $\cH_q(\Sigma_m)$, where supp$(\bbla) = \{\la\}$.
Consequently, the cell module $W(e_\la^{\iyng{1}}) \cong S^\la \in \cH_q(\Sigma_m)$.
\item Suppose that $m=1$. 
Let $\la := \bblambda(\iyng{1}) \in \Pi_d$.
The element $C^\bbla_{\bbT, \bbT'} \in K\wr\cH(d)$ is given by 
\[
C^\bbla_{\bbT, \bbT'} = \sum_{x \in \Sigma_d} H_{d^{-1}_{\tabs_{\bbT}}}
z_\la^{\iyng{1}}
H_{d_{\tabs_{\bbT'}}},
\]
and hence the cell module $W(e^\la_{\iyng{1}})$ is isomorphic to the Specht module with respect to $\la$ in $\cH_{(0,(1+q)^2)}(\Sigma_d)$.
\end{enua}
\end{lem}
\begin{proof}
For part (a), each $\bbla \in \Omega$ is of the form $e_{\la}^{\iyng{1}}$ for some $\la \vdash m$. Then, $\overline{\bbla} = \la \in \Pi_m$, and hence
$\mu(\bbla) = (1) \vDash 1$.
Next, $M(\bbla) = \Sigma_1 \times \std(\iyng{1}) \times \std(\la)$.
Therefore, each $\bbT \in M(\bbla)$ is of the form $(e, \iyoung{1}, \tabt_\bbT)$ where $\tabt_\bbT \in \std(\la)$.
Observe that $\tabs_\bbT$ has a single component and hence
\[
\zeta_{\bbT, \bbT'}(e)
=
\zeta^{\la}_{\iyng{1}}(e) = 1 
\]
since $z^\la_{\iyng{1}} = 1$.
Therefore,
\[
C^{e_\la^{\iyng{1}}}_{\bbT,\bbT'}
=
\zeta_{\bbT, \bbT'}(e)
C^\la_{\tabt_{\bbT}, \tabt_{\bbT'}}
H_e
=
C^\la_{\tabt_{\bbT}, \tabt_{\bbT'}},
\]
which coincides with the Murphy basis element for $\cH_{q}(\Sigma_m)$.

For part (b),  each $\bbla \in \Omega$ is of the form $e_{\iyng{1}}^\la$ for some $\la \vdash d$.
Thus, $\overline{\bbla} = (\iyng{1}, \dots, \iyng{1}) \in \Pi_1^d$, and hence
$\mu(\bbla) = (d) \vDash d$.
One gets $M(\bbla) = \Sigma_1 \times \std(\la) \times \prod_{i=1}^d\std(\iyng{1})$.
That is, each $\bbT \in M(\bbla)$ is of the form $(e, \tabs_\bbT, (\iyoung{1}, \dots, \iyoung{1}))$ where $\tabs_\bbT \in \std(\la)$.
Again, $\tabs_\bbT$ has one single component.
Thus,
\[
C^{e^\la_{(1)}}_{\bbT,\bbT'}
=
\sum_{x\in \Sigma_d}
\zeta^{\iyng{1}}_\la(x)
(C^{(1)}_{\iyoung{1}, \iyoung{1}})^{\otimes d} 
H_{d^{-1}_{\tabs_{\bbT}}}H_x H_{d_{\tabs_{\bbT'}}}
=
H_{d^{-1}_{\tabs_{\bbT}}}
\Big(
\sum_{x\in \Sigma_d}
\zeta^{\iyng{1}}_\la(x)
H_x 
\Big)
H_{d_{\tabs_{\bbT'}}}.
\]
Since $m=1$, it follows from \eqref{eq:z1} that $B\wr\cH(d) \cong \cH_{(0, (1+q)^2)}(\Sigma_d)$, and hence $C^{e^\la_{(1)}}_{\bbT,\bbT'}$ coincides with the full symmetrizer basis element therein.
\end{proof}

\section{Cellularity of Hu Algebras}
In this section, we show that the Hu algebra $B\wr\cH(2)$ is a cellular algebra.
Although our construction is expected to hold for $d>2$, a rigorous proof is highly technical and falls outside the scope of this work.

\subsection{Cellular Basis Elements for Hu Algebras}
\begin{lem}\label{lem:d=2}
Suppose that $d=2$.
The elements $C^\bbla_{\bbT, \bbT'}$ in the Hu algebra $\cH_q(\Sigma_m)\wr\cH(2)$ are given by the following:
\[
C^\bbla_{\bbT, \bbT'} =
\begin{cases}
C^{\la}_{\tabt_\bbT^{1},\tabt_{\bbT'}^{1}} \otimes C^{\la}_{\tabt_\bbT^{2},\tabt_{\bbT'}^{2}}
+
\sqrt{f_\la}^{-1}
\big(
C^{\la}_{\tabt_\bbT^{1},\tabt_{\bbT'}^{2}} \otimes C^{\la}_{\tabt_\bbT^{2},\tabt_{\bbT'}^{1}}
\big) 
H_{1}
&
\textup{if } \bbla = e_\la^{\iyng{2}};
\\
C^{\la}_{\tabt_\bbT^{1},\tabt_{\bbT'}^{1}} \otimes C^{\la}_{\tabt_\bbT^{2},\tabt_{\bbT'}^{2}}
-
\sqrt{f_\la}^{-1}
\big(
C^{\la}_{\tabt_\bbT^{1},\tabt_{\bbT'}^{2}} \otimes C^{\la}_{\tabt_\bbT^{2},\tabt_{\bbT'}^{1}}
\big) 
H_{1}
&
\textup{if } \bbla = e_\la^{\iyng{1,1}};
\\
H_{w_{\bbT}^{-1}}
\big(
C^{\pi^{i}}_{\tabt_\bbT^{1},\tabt_{\bbT'}^{1}} \otimes C^{\pi^{j}}_{\tabt_\bbT^{2},\tabt_{\bbT'}^{2}}
\big) 
H_{w_{\bbT'}}
&\textup{if } \bbla = [\pi^{i},\pi^{j}].
\end{cases}
\]
where $i < j$ in the third case.
\end{lem}
\begin{proof}
First, consider the case $\bbla = e_\la^\nu$ for some $\nu \vdash 2$.
From \cref{eg:d=2'}(b), any cellular basis element is of the form
\[
C^{e_\la^\nu}_{
(1, \tabs_R^\nu, (\tabt_{\bbT}^{1},\tabt_{\bbT}^{2}))
, (1, \tabs_R^\nu, (\tabt_{\bbT'}^{1},\tabt_{\bbT'}^{2}))
}
=
\sum_{x \in \Sigma_2} 
\zeta^{\la}_{\nu}(x)
(C^\la_{\tabt_{\bbT}^{1},\tabt_{\bbT'}^{x(1)}}
\otimes C^\la_{\tabt_{\bbT}^{2},\tabt_{\bbT'}^{x(2)}})
H_x.
\]
The first two cases of the assertion follow from \eqref{eq:zeta}.
Next, consider the case $\bbla = [\pi^{i}, \pi^{j}]$.
We identify $\bbT = 
(w_\bbT, (\tabs_R^{\iyng{1}},\tabs_R^{\iyng{1}}), (\tabt_{\bbT}^{1},\tabt_{\bbT}^{2}))$ with $(w_\bbT, \tabt_{\bbT}) \in \Sigma_2 \times (\std(\pi^i)\times \std(\pi^j))$.
Any cellular basis element is of the form
\[
C^{[\pi^i, \pi^j]}_{(w_{\bbT},\tabs_\bbT, \tabt_{\bbT}), (w_{\bbT'}, \tabs_{\bbT'}, \tabt_{\bbT'})}
=
z^{\pi^{i}}_{\iyng{1}}(e)
z^{\pi^{j}}_{\iyng{1}}(e)
H_{w_{\bbT}^{-1}}
(C^{\pi^{i}}_{\tabt_{\bbT}^{1},\tabt_{\bbT'}^{1}}
\otimes C^{\pi^{j}}_{\tabt_{\bbT}^{2},\tabt_{\bbT'}^{2}})
H_e
H_{w_{\bbT'}}.
\]
Therefore, the last part of the assertion follows since  $z_{\iyng{1}}^\la = 1$ for any $\la\vdash m$.
\end{proof}
\subsection{Proof of the Cellularity Theorem}

\begin{thm}\label{thm:cellularity}
The Hu algebra $B \wr \cH(2)$ is a cellular algebra with cell datum $(\Omega,M,C,*)$, where
$\Omega = \Pi_{\Pi_m}(2)$ (see \cref{defn:lambda}), 
$M(\bbla)$ is defined as in \cref{defn:m-set},
$C_{\bbT, \bbT'}^\bbla$ is defined as in \cref{lem:d=2}(a),
and the involutive anti-automorphism $\ast$ is from \cref{gHuast}.
\end{thm}
\begin{proof}   
    For C1),
    by the PBW basis theorem \cite[Theorem 3.3.1, Section 4.1]{LNX}, the set 
    \[
    \{
    (C^{\la^{1}}_{\tabs_1,\tabt_1}\otimes C^{\la^{2}}_{\tabs_2,\tabt_2})H_w
    ~|~
    \la^{i} \vdash m,\ \tabs_i, \tabt_i \in \std(\la^{i}), \ w \in \Sigma_2
    \}
    \]
    is a $K$-basis of $B\wr\cH(d)$. 
    By \cref{lem:d=2}, the set $\{C^\bbla_{\bbT, \bbT'} ~|~ \bbT, \bbT' \in M(\bbla)\}$ forms a $K$-basis of $B\wr \cH(2)$ since $\ch K \neq 2$.

    For C2), 
    it suffices to check $(C^{\bbla}_{\bbT, \bbT'})^* = C^{\bbla}_{\bbT', \bbT}$ for each of the three cases  in \cref{lem:d=2}. 
    Within this proof, we write $x_i := \tabt_\bbT^{i}, y_i := \tabt_{\bbT'}^{i}$ for brevity.
    For the first two cases, the equality holds since
\begin{equation}
\begin{split}    
(C^\bbla_{(e,\tabs_R^\nu,(x_1,x_2)), (e,\tabs_R^\nu,(y_1,y_2))})^\ast
&=
(
C^{\la}_{x_1,y_1} \otimes C^{\la}_{x_2,y_2}
)^\ast 
\pm
\sqrt{f_\la}^{-1}
H_{1}^\ast
(
C^{\la}_{x_1,y_2} \otimes C^{\la}_{x_2,y_1}
) 
^\ast
\\
&=
(C^{\la}_{y_1,x_1} \otimes C^{\la}_{y_2,x_2})
\pm
\sqrt{f_\la}^{-1}
(
  C^{\la}_{y_1,x_2} 
  \otimes
 C^{\la}_{y_2,x_1}
) 
H_{1}
\\
&= C^\bbla_{(e,\tabs_R^\nu,(y_1,y_2)), (e,\tabs_R^\nu,(x_1,x_2))}.
\end{split}
\end{equation}
For the third case, consider the first subcase when $w_\bbT= w_{\bbT'}$.
It is easy to check that $(C^\bbla_{\bbT, \bbT'})^\ast = C^\bbla_{\bbT', \bbT}$ since no $H_1$'s are involved. 
For the second subcase when $w_\bbT \neq w_{\bbT'}$, one has
\begin{equation}
\begin{split}
(C^\bbla_{(s_1,\tabs_R^\nu,(x_1,x_2)), (e,\tabs_R^\nu,(y_1,y_2))})^\ast
&=
H_1^\ast (C^{\pi^{j}}_{x_2, y_2} \otimes C^{\pi^{i}}_{x_1, y_1})^\ast
\\
&=
(C^{\pi^{i}}_{y_1, x_1} \otimes C^{\pi^{j}}_{y_2, x_2})H_1 
\\
&=C^\bbla_{(e,\tabs_R^\nu,(y_1,y_2)), (s_1,\tabs_R^\nu,(x1,x2))},
\end{split}
\end{equation}
and hence C2) is verified by symmetry.
For C3), we need to show that for all $a \in B \wr \cH(2)$ there exist scalars 
$r_a(\bbS,\bbT)\in K$ such that
\begin{equation}
a\,C^{\bbla}_{\bbT,\bbT'}
\;\equiv\;
\sum_{\bbS \in M(\bbla)} r_a(\bbS,\bbT)\,
C^{\bbla}_{\bbS,\bbT'}
\pmod{\cala_{<\bbla}},
\end{equation}
where $\cala_{<\bbla} = \textup{Span}_K
\{ C^{\bbmu}_{\bbT,\bbT'}
\mid \bbmu < \bbla,\, \bbT, \bbT' \in  M(\bbmu)\}$.
First, consider the case $a=b_1\otimes b_2 \in B\otimes B$.
By the cellularity of $B$ in each tensor factor, one has, for any $\la^{i} \vdash m$, $\tabt_i, \tabt'_i \in \std(\la^{i})$:
\begin{equation}
\begin{split}
&(b_1 \otimes b_2)
\bigl(C^{\lambda^{1}}_{\tabt_1,\tabt'_1}
\otimes
C^{\lambda^{2}}_{\tabt_2,\tabt'_2}\bigr)
\\
&\quad\equiv
\sum_{(\tabs_1,\tabs_2)\in\std(\la^{1})\times \std(\la^{2})}
r_{b_1}(\tabs_1,\tabt_1)
r_{b_2}(\tabs_2,\tabt_2)
\bigl(C^{\lambda^{1}}_{\tabs_1,\tabt'_1}
\otimes
C^{\lambda^{2}}_{\tabs_2,\tabt'_2}\bigr)
\pmod{B_{<\lambda^{1}}\otimes B_{<\lambda^{2}}}.
\end{split}
\end{equation}
Note that the lower terms in $B_{<\lambda^{1}}\otimes B_{<\lambda^{2}}$ correspond to elements of the form $C^{\bbla'}_{\bullet,\bullet}$ with supp$(\bbla')$ obtained by replacing certain partitions in supp$(\bbla)$ by smaller partitions with respect to the dominance order on $\Pi_m$.
Thus it follows from \cref{def:domPO} that both $B_{<\lambda^{1}}\otimes B_{<\lambda^{2}}$ and $(B_{<\lambda^{1}}\otimes B_{<\lambda^{2}})H_1$ lie in $\cA_{<\bbla}$.

Next, define $r_{b_1\otimes b_2}(\bbS,\bbT) := 
r_{b_1}(\tabt_\bbS^{1}, \tabt_\bbT^{1}) r_{b_2}(\tabt_\bbS^{2}, \tabt_\bbT^{2})$.
Consider the highest order terms in the first two cases in \cref{lem:d=2}(a). One has
\begin{equation}
\begin{split}
&a C^\bbla_{(e,\tabs^{\nu}_R,\tabt_\bbT), (e,\tabs^{\nu}_R,\tabt_{\bbT'})}
\\
&\equiv
\sum_{\substack{\bbS = (e,\tabs^{\nu}_R, (z_1,z_2))\\ \in M(\bbla)}}
r_a(\bbS,\bbT)
C^{\la}_{z_1,y_2} \otimes C^{\la}_{z_2,y_1}
\pm
\sum_{\substack{\bbS = (e,\tabs^{\nu}_R, (z_1,z_2))\\ \in M(\bbla)}}
r_a(\bbS,\bbT)
\sqrt{f_\la}^{-1}
(
C^{\la}_{z_1,y_1} \otimes C^{\la}_{z_2,y_2}
)
H_{1}
\\
&\equiv 
\sum_{\bbS = (e,\tabs^{\nu}_R,(z_1,z_2))
\in M(\bbla)}
r_a(\bbS, \bbT) \big(
C^{\la}_{z_1,y_1} \otimes C^{\la}_{z_2,y_2} 
\pm
\sqrt{f_\la}^{-1}
(
C^{\la}_{z_1,y_2} \otimes C^{\la}_{z_2,y_1}
)
H_{1}
\big)
\\
&\equiv 
\sum_{\bbS \in M(\bbla)}
r_a(\bbS, \bbT) 
C^\bbla_{\bbS, \bbT'}
\pmod{\cala_{<\bbla}}.
\end{split}
\end{equation}
The second equivalence  relies on the fact that each of the cellular basis elements we constructed is of the form 
$
 C^\bullet_{x_1\bullet}\otimes C^\bullet_{x_2,\bullet}
 \pm 
\sqrt{f_\la}^{-1}
(C^\bullet_{x_1\bullet}\otimes C^\bullet_{x_2,\bullet}) H_1$, which results in the coincidence that the coefficients in both summations are given by the same function $r_a(\bbS,\bbT)$.
We omit the verification for the last case in \cref{lem:d=2} as it is similar but simpler.  

Finally, consider the case  $a=H_1$. 
We will construct the function $r_{H_1}(\bbS, \bbT)$ for all $\bbS, \bbT \in M(\bbla)$.
Suppose that $\bbla = e_\nu, \ \nu\vdash 2$. Set
\begin{equation}\label{eq:H1a}
    r_{H_1}(\bbS, (e, \tabs_\bbT, (x_1, x_2))) = 
    \begin{cases}
        \sqrt{f_\la} &\textup{if }\bbS=(e,\tabs_\bbT, (x_2,x_1)), \ \nu = (2);
        \\
        -\sqrt{f_\la} &\textup{if }\bbS=(e,\tabs_\bbT, (x_2,x_1)), \ \nu = (1,1);
        \\
        0 &\textup{otherwise}.
    \end{cases}
\end{equation}
Then, by \cref{lem:d=2} and \cref{cor:zmaction}, one has
\[
\begin{split}
H_1 C^\bbla_{\bbT, \bbT'} 
&= (C^\la_{x_2,y_2} \otimes C^\la_{x_1,y_1})H_1 
\pm \sqrt{f_\la}^{-1} H^2_1 (C^\la_{x_2,y_1}\otimes C^\la_{x_1,y_2}) 
\\
&= \pm\sqrt{f_\la} C^\bbla_{(e,\tabs_\bbT,(x_2,x_1)),\bbT'}
= \sum\nolimits_{\bbS \in M(\bbla)} r_{H_1}(\bbS,\bbT) C^\bbla_{\bbS,\bbT'}.
\end{split}
\]
Suppose that $\bbla = [\la, \la']$ for some $\la\neq \la' \vdash m$. Set
\begin{equation}\label{eq:H1b}
    r_{H_1}(\bbS, \bbT) = 
    \begin{cases}
        1 &\textup{if }\bbS =(s_1 w_\bbT, \tabs_\bbT, \tabt_\bbT), \ s_1 w_\bbT > w_\bbT;
        \\
        f_{\la, \la'} &\textup{if }\bbS =(s_1 w_\bbT, \tabs_\bbT, \tabt_\bbT), \ s_1 w_\bbT < w_\bbT;
        \\
        0 &\textup{otherwise}.
    \end{cases}    
\end{equation}
Again by \cref{lem:d=2}, a direct case-by-case computation shows that
\[
H_1 C^\bbla_{(e,\tabs_\bbT,\tabt_\bbT),\bbT'} = C^\bbla_{(s_1,\tabs_\bbT,\tabt_\bbT),\bbT'},
\quad
H_1 C^\bbla_{(s_1,\tabs_\bbT,\tabt_\bbT),\bbT'} = f_{\la,\la'}C^\bbla_{(e,\tabs_\bbT,\tabt_\bbT),\bbT'},
\]
where we used the fact that $f_{\la,\la'} = f_{\la',\la}$.
In other words,
$
H_1 C^\bbla_{\bbT, \bbT'} 
= \sum_{\bbS \in M(\bbla)} r_{H_1}(\bbS,\bbT) C^\bbla_{\bbS,\bbT'}$.
That is, C3) is verified for all the generators of $\cala(m)$, and hence holds true for all elements in $\cala(m)$. 
The Hu algebra is indeed cellular.
\end{proof}

\subsection{Robinson--Schensted Correspondence}

The property C1) gives a bijection on the index sets which can be regarded as a Robinson--Schensted correspondence for $\Sigma_m \wr \Sigma_2$, which aligns with Okada's generalized Robinson-Schensted correspondence \cite{O90}.
Note that this bijection
already differs from neither the one obtained from a decomposition of a Steinberg variety in  \cite{HL26}, 
nor the exotic Robinson-Schensted correspondence in \cite{NRS21} for $d=2$.

\begin{eg}
Let $m=d=2$. 
Fix a total order on $\Pi_2$ by letting $\pi^1 = \iyng{1,1}$ and $\pi^2= \iyng{2}$.
Recall that 
$
\Pi_{\Pi_2}(2) = \{ 
e_{(2)}^{(2)},  e_{(2)}^{(1,1)}, e_{(1,1)}^{(2)}, e_{(1,1)}^{(1,1)},
[(1,1),(2)]
\}$.
It is convenient to write
\[
\std(\bbla) := 
\std(\bbla(\iyng{1,1}),\bbla(\iyng{2}))
\times
(\std(\overline{\bbla}(1))
\times 
\std(\overline{\bbla}(2))),
\]
and hence
\[
\begin{split}
&\std(e_{(2)}^{(2)}) = \{(\iyoung{12}, (\iyoung{12}),\iyoung{12}))\},
\quad
\std(e_{(1,1)}^{(2)}) = \{(\iyoung{12}, (\iyoung{1,2},\iyoung{1,2}))\},
\\
&\std(e_{(2)}^{(1,1)}) = \{(\iyoung{1,2}, (\iyoung{12},\iyoung{12}))\},
\quad
\std(e_{(1,1)}^{(1,1)}) = \{(\iyoung{1,2}, (\iyoung{1,2},\iyoung{1,2}))\},
\\
&\std([\iyng{2},\iyng{1,1}]) = \{(
(\iyoung{1}, \iyoung{2}), (\iyoung{1},\iyoung{1})),
(\iyoung{2}, \iyoung{1}), (\iyoung{1},\iyoung{1}))
\}.
\end{split}
\]
Let $s$ and $u$ be the generators of the first and second $\Sigma_2$ in $\Sigma_2 \wr \Sigma_2$, respectively. 
Identify $\Sigma_2 \wr \Sigma_2$ with the subgroup of $\Sigma_4$ generated by $s_1 = (1~2), s_3 = (3~4)$, and $t = (3~4~1~2)$ by $(s,1) \mapsto s_1, (1,s) \mapsto s_3$, and $u \mapsto t$. 
One can write down an evident bijection using \cref{lem:d=2}. That is,
\[
\begin{array}{cccc}
     1\mapsto
     \big(
     (\iyoung{12}, (\iyoung{1,2},\iyoung{1,2})), 
     (\iyoung{12}, (\iyoung{1,2},\iyoung{1,2}))
     \big),
     &s_1s_3
     \mapsto
     \big(
     (\iyoung{12}, (\iyoung{12}),\iyoung{12})),
     (\iyoung{12}, (\iyoung{12}),\iyoung{12}))
     \big),
     \\
     t\mapsto
     \big(
     (\iyoung{1,2}, (\iyoung{1,2},\iyoung{1,2})),
     (\iyoung{1,2}, (\iyoung{1,2},\iyoung{1,2}))
     \big),
     &s_1s_3t\mapsto
     \big(
     (\iyoung{1,2}, (\iyoung{12},\iyoung{12})),
     (\iyoung{1,2}, (\iyoung{12},\iyoung{12}))
     \big),
     \\
     s_1\mapsto
     \big(
     (\iyoung{2}, \iyoung{1}), (\iyoung{1},\iyoung{1})),
     (\iyoung{2}, \iyoung{1}), (\iyoung{1},\iyoung{1}))
     \big),
     &s_3\mapsto
     \big(
     (\iyoung{1}, \iyoung{2}), (\iyoung{1},\iyoung{1})),
     (\iyoung{1}, \iyoung{2}), (\iyoung{1},\iyoung{1}))
     \big),
     \\
     s_1t\mapsto 
     \big(
     (\iyoung{2}, \iyoung{1}), (\iyoung{1},\iyoung{1})),
     (\iyoung{1}, \iyoung{2}), (\iyoung{1},\iyoung{1}))
     \big),
     &s_3t\mapsto 
     \big(
     (\iyoung{1}, \iyoung{2}), (\iyoung{1},\iyoung{1})),
     (\iyoung{2}, \iyoung{1}), (\iyoung{1},\iyoung{1}))
     \big).
\end{array}
\]
On the other hand, if one chooses $\pi^1 = \iyng{2}$ and $\pi^2 = \iyng{1,1}$ instead, then the evident bijection is exactly the one given in \cite{O90}.
\end{eg}
\section{Cell Modules}
In this section, we consider two types of cell modules separately.
The first type are the cell modules $W(\bbla)$ where $\bbla = e_\la^\nu$ for some $\la \vdash m$, $\nu \vdash 2$.
The second type corresponds to $\bbla = [\la, \la']$ for some $\la \neq \la'$ in $\Pi_m$. We show that these cell modules are isomorphic to the Specht modules defined by Hu and further compute their canonical bilinear forms.
\subsection{Wreath Modules}\label{sec:wreathmod}
We will show that these cell modules are isomorphic to the Specht module for $\cala(m)$ introduced by Hu in \cite{hu02}, which we paraphrase as follows.
\begin{defn}[Specht modules]
    For each $\bbla \in \Omega = \{ e_\la^\nu ~|~ \la \vdash m, \ \nu\vdash 2\} \sqcup \{[\la, \la'] ~|~ \la, \la' \vdash m, \ \la \neq \la'\}$,
    denote its corresponding {\em Specht module} by
    \[
    S^\bbla := \begin{cases}
\Span_{K}\{\sqrt{f_{\la}}(x\otimes y)+(x\otimes y)H_1\mid x\in S^{\la},y\in S^{\la}\}    
&\textup{if }\bbla = e_\la^{\iyng{2}};
\\
\Span_{K}\{\sqrt{f_{\la}}(x\otimes y)-(x\otimes y)H_1\mid x\in S^{\la},y\in S^{\la}\}    
&\textup{if }\bbla = e_\la^{\iyng{1,1}};
\\
\Ind^\cala_{B\otimes B}(S^{\la}\otimes S^{\la'})
&\textup{if }\bbla = [\la, \la'].
\end{cases}
    \]
\end{defn}
It is convenient to describe these Specht modules using the language of the \emph{wreath modules} \cite{LNXII}. 
Below we summarize the construction related to the Hu algebras, following \cite[Section 6.6]{LNXII}.

Suppose that $M \in \cH_q(\Sigma_m)$-mod such that $z_m$ acts on $M\otimes M$ by a scalar $f_\la \in K$, $N$ is a module over the generic Hecke algebra $\cH_{(0,f_\la)}(\Sigma_d)$, then, 
$M \wr N := M^{\otimes d} \otimes N$
affords a $B\wr \cH(d)$-module structure given by
\[
b \cdot(m\otimes n) = (b \cdot m)\otimes n,
\quad
H_i \cdot (m\otimes n) = m\otimes (T_i \cdot n)
\]
for $b \in B^{\otimes d}$, $m\in M^{\otimes d}$, $n \in N$, and $w \in \Sigma_d$.
Therefore, $S^{e_\la^\nu} \cong S^\la \wr S^\nu_{(0,f_\la)}$, where $S^\nu_{(0,f_\la)}$ is the Specht module over $\cH_{(0,f_\la)}(\Sigma_d)$.
\begin{thm}\label{thm:CellIsSpecht}
For the Hu algebras, the Specht modules coincide with the cell modules arising from the cell datum in \cref{thm:cellularity}. That is,
\[
W(\bbla) \cong S^\bbla =\begin{cases}
    S^\la \wr S^\nu_{(0, f_\la)} 
    &\textup{if }\bbla = e_\la^\nu;
    \\
    \Ind_{B\otimes B}^{\cala(m)}(S^\la \otimes S^{\la'})
    &\textup{if }\bbla = [\la, \la'].
\end{cases}
\]
\end{thm}
\begin{proof}
When $\bbla = [\la, \la']$, the induced module $S^\bbla = \mathcal{A}(m) \otimes_{B\otimes B} (S_m^\lambda \otimes S_m^{\lambda'})$ 
 has a natural $K$-basis 
 \[
 \{ H_w \otimes (C^\la_{\tabt_1} \otimes C^{\la'}_{\tabt_2}) ~|~ w\in \Sigma_2, \ 
 \tabt_1 \in \std(\la), \ \tabt_2 \in \std(\la') \}.
 \]
We will show that the following assignment is an isomorphism:
\begin{equation}\label{eq:celliso}
W(\bbla) \to S^\bbla,
~ 
C_\bbT \mapsto
\begin{cases}
C^\la_{\tabt_\bbT^1} \otimes C^\la_{\tabt_\bbT^2} \otimes 1
&\textup{if }\bbla=e_\la^\nu, \bbT = (e, \tabs_R^\nu, (\tabt_\bbT^1, \tabt_\bbT^2));
\\
H_{w_\bbT} \otimes C^\la_{\tabt_\bbT^1} \otimes C^{\la'}_{\tabt_\bbT^2} 
&\textup{if }\bbla=[\la, \la'], \bbT = (w_\bbT, \tabs_R^\nu, (\tabt_\bbT^1, \tabt_\bbT^2)).
\end{cases}
\end{equation}
Recall that the $\cala(m)$-action on the cell module is given by, for $a \in \cala(m)$:
\begin{equation*}
    aC_\bbT^\bbla=\sum\nolimits_{\bbS\in M(\bbla)} r_a(\bbS,\bbT) C_{\bbT}^\bbla.
\end{equation*}
Since $r_{b_1\otimes b_2}(\bbS, \bbT) = r_{b_1}(\tabt_\bbS^{1}, \tabt_\bbT^{1}) r_{b_2}(\tabt_\bbS^{2}, \tabt_\bbT^{2})$ for all $b_i \in B$,
the $B\otimes B$-action on the cell module agree with the one on the Specht module.

On the other hand, the $H_1$-action on the cell modules are given by \eqref{eq:H1a} when $\bbla = e_\la^\nu$ and by \eqref{eq:H1b} when $\bbla = [\la, \la']$.
Indeed, they coincide with the $H_1$-action on the Specht modules. The theorem is proved.
\end{proof}
\subsection{Cell Modules, First Type}
 Since $B = \cH_q(\Sigma_m)$ is cellular, the corresponding canonical bilinear form $\phi_\la$ satisfies that
\begin{equation}
    C^\la_{\tabs_1, \tabs_2}
        C^\la_{\tabt_1, \tabt_2} 
        \equiv
            \phi_\la(C_{\tabs_2}^\la, C_{\tabt_1}^\la)C^\la_{\tabs_1, \tabt_2}
\pmod{B_{<\la}}
\end{equation}
for any $\tabs_i, \tabt_i \in \std(\la)$. 

Suppose that $\bbla = e_\la^\nu$ for some $\la \vdash m$, $\nu \vdash 2$. By \cref{lem:d=2}, an arbitrary cellular basis element is of the form
\[
C^\bbla_{\bbT, \bbT'} 
= 
C^{\la}_{\tabt_\bbT^{1},\tabt_{\bbT'}^{1}} \otimes C^{\la}_{\tabt_\bbT^{2},\tabt_{\bbT'}^{2}}
\pm
\sqrt{f_\la}^{-1}
\big(
C^{\la}_{\tabt_\bbT^{1},\tabt_{\bbT'}^{2}} \otimes C^{\la}_{\tabt_\bbT^{2},\tabt_{\bbT'}^{1}}
\big) 
H_{1}.
\]
\begin{prop}\label{prop:phibbla1}
Suppose that $\bbla = e_\la^\nu$ for some $\la \vdash m$, $\nu \vdash 2$.
Then,
\[
    C^\bbla_{\bbT, \bbT}    C^\bbla_{\bbT', \bbT'}
    \equiv
    2 \phi_\la(C^\la_{\tabt_{\bbT}^1}, C^\la_{\tabt_{\bbT'}^1})
    \phi_\la(C^\la_{\tabt_{\bbT}^2}, C^\la_{\tabt_{\bbT'}^2}) C^\bbla_{\bbT, \bbT'}
    \pmod{\cala_{<\bbla}}.
\]
As a result, $\phi_\bbla(C^\bbla_{\bbT},    C^\bbla_{\bbT'}) =     2 \phi_\la(C^\la_{\tabt_{\bbT}^1}, C^\la_{\tabt_{\bbT'}^1})
    \phi_\la(C^\la_{\tabt_{\bbT}^2}, C^\la_{\tabt_{\bbT'}^2})$.
\end{prop}
\begin{proof}
Within this proof, Write $x_i := \tabt_\bbT^{(i)}, y_i := \tabt_{\bbT'}^{(i)}$ for brevity. 
Recall that both $B_{<\lambda}\otimes B_{<\lambda}$ and $(B_{<\lambda}\otimes B_{<\lambda})H_1$ lie in $\cA_{<\bbla}$ thanks to \cref{def:domPO}.
Hence, it follows from \cref{cor:zmaction} that
\begin{equation}
\begin{split}
    C^\bbla_{\bbT, \bbT}    C^\bbla_{\bbT', \bbT'}
    &=
C^{\la}_{x_1,x_1}C^{\la}_{y_1,y_1} \otimes C^{\la}_{x_2,x_2}C^{\la}_{y_2,y_2}
+
\sqrt{f_\la}^{-2}
(
C^{\la}_{x_1,x_2}C^{\la}_{y_2,y_1} \otimes C^{\la}_{x_2,x_1}C^{\la}_{y_1,y_2}
) 
H_{1}^2
\\
&+
\sqrt{f_\la}^{-1}
(
C^{\la}_{x_1,x_2}C^{\la}_{y_2,y_2} \otimes C^{\la}_{x_2,x_1}C^{\la}_{y_1,y_1}
+
C^{\la}_{x_1,x_1}C^{\la}_{y_1,y_2} \otimes C^{\la}_{x_2,x_2}C^{\la}_{y_2,y_1}
) 
H_1
\\
&\equiv
\phi_\la(C^\la_{x_1},C^\la_{y_1}) C^{\la}_{x_1,y_1} \otimes \phi_\la(C^\la_{x_2},C^\la_{y_2}) C^{\la}_{x_2,y_2}
\\
&\quad+
\phi_\la(C^\la_{x_2},C^\la_{y_2})C^{\la}_{x_1,y_1} \otimes \phi_\la(C^\la_{x_1},C^\la_{y_1})C^{\la}_{x_2,y_2} 
\\
&\quad \pm
\sqrt{f_\la}^{-1}
(\phi_\la(C^\la_{x_2},C^\la_{y_2}) C^{\la}_{x_1,y_2} \otimes \phi_\la(C^\la_{x_1},C^\la_{y_1})C^{\la}_{x_2,y_1})H_1
\\
&\quad \pm
\sqrt{f_\la}^{-1}
(\phi_\la(C^\la_{x_1},C^\la_{y_1})C^{\la}_{x_1,y_2} \otimes \phi_\la(C^\la_{x_2},C^\la_{y_2})C^{\la}_{x_2,y_1})H_1
\\
&\equiv 2 \phi_\la(C^\la_{x_1}, C^\la_{y_1})
    \phi_\la(C^\la_{x_2}, C^\la_{y_2}) C^\bbla_{\bbT, \bbT'}
    \pmod{\cala_{<\bbla}}.
\end{split}    
\end{equation}
The second assertion of the lemma follows from the definition of the canonical bilinear form.
\end{proof}
\begin{cor}\label{rad1}
    Suppose that $\bbla = e_\la^\nu$ for some $\la \vdash m$, $\nu \vdash 2$.
Then,
\[
\rad(\bbla) \cong  \left( \rad(\phi_\lambda) \otimes S^\lambda \right) \oplus \left( S^\lambda \otimes \rad(\phi_\lambda) \right).
\]
\end{cor}
\begin{proof}
    One has $\rad(\bbla) = \{C^\bbla_\bbT ~|~ \phi_\bbla(C^\bbla_\bbT, C^\bbla_{\bbT'}) = 0 \textup{ for all }\bbT' \in M(\bbla)\}$.
    Since $\ch K \neq 2$, it follows from  \cref{prop:phibbla1} that $ \phi_\bbla(C^\bbla_\bbT, C^\bbla_{\bbT'}) = 0$ if and only if either $C^\la_{\tabt_\bbT^1} \in \rad(\phi_\la)$ or $C^\la_{\tabt_\bbT^2} \in \rad(\phi_\la)$. 
    The corollary is proved. 
\end{proof}
\subsection{Cell Modules, Second Type}
Next, suppose $\bbla = [\la, \la']$ for some $\la \neq \la'$ in $\Pi_m$. 
By \cref{lem:d=2}, an arbitrary cellular basis element is of the form
\[
C^\bbla_{\bbT, \bbT'} 
= 
H_{w_{\bbT}^{-1}}
\big(
C^{\la}_{\tabt_\bbT^{1},\tabt_{\bbT'}^{1}} \otimes C^{\la'}_{\tabt_\bbT^{2},\tabt_{\bbT'}^{2}}
\big) 
H_{w_{\bbT'}}.
\]
\begin{prop}\label{prop:phibbla2}
Suppose that $\bbla = [\la, \la']$ for some $\la \neq \la'$ in $\Pi_m$. 
Then,
\[
    C^\bbla_{\bbT, \bbT}    C^\bbla_{\bbT', \bbT'}
    \equiv
    \begin{cases}
    \phi_\la(C^\la_{\tabt_{\bbT}^1}, C^\la_{\tabt_{\bbT'}^1})
    \phi_{\la'}(C^{\la'}_{\tabt_{\bbT}^2}, C^{\la'}_{\tabt_{\bbT'}^2}) 
    C^\bbla_{\bbT, \bbT'}
    \pmod{\cala_{<\bbla}}
    &\textup{if }w_\bbT = w_{\bbT'} = e;
    \\
    f_{\la,\la'}
    \phi_\la(C^\la_{\tabt_{\bbT}^1}, C^\la_{\tabt_{\bbT'}^1})
    \phi_{\la'}(C^{\la'}_{\tabt_{\bbT}^2}, C^{\la'}_{\tabt_{\bbT'}^2})  
    C^\bbla_{\bbT, \bbT'}
    \pmod{\cala_{<\bbla}}
    &\textup{if }w_\bbT = w_{\bbT'} = s_1;
    \\
    0
    \pmod{\cala_{<\bbla}}
    &\textup{if }w_\bbT \neq w_{\bbT'}.
    \end{cases}
\]
As a result, the canonical bilinear form on $W([\la, \la'])$ is given by
\[
\phi_{[\la,\la']}(C^\bbla_{\bbT},    C^\bbla_{\bbT'})
    =
    \begin{cases}
    \phi_\la(C^\la_{\tabt_{\bbT}^1}, C^\la_{\tabt_{\bbT'}^1})
    \phi_{\la'}(C^{\la'}_{\tabt_{\bbT}^2}, C^{\la'}_{\tabt_{\bbT'}^2}) 
    &\textup{if }w_\bbT = w_{\bbT'} = e;
    \\
    f_{\la,\la'}
    \phi_\la(C^\la_{\tabt_{\bbT}^1}, C^\la_{\tabt_{\bbT'}^1})
    \phi_{\la'}(C^{\la'}_{\tabt_{\bbT}^2}, C^{\la'}_{\tabt_{\bbT'}^2}) 
    &\textup{if }w_\bbT = w_{\bbT'} = s_1;
    \\
    0
    &\textup{if }w_\bbT \neq w_{\bbT'}.
    \end{cases}
\]
\end{prop}
\begin{proof}
Assume that $\la = \pi^i$, $\la' = \pi^j$ with $i < j$.
Suppose that $w_\bbT \neq w_{\bbT'}$.
We may assume that 
$\bbT = (s_1, \tabs_\bbT, (x_1, x_2))$ and $\bbT' = (e, \tabs_{\bbT'}, (y_1, y_2))$. Thus,
\begin{equation}
    C_{\bbT,\bbT}^\bbla C_{\bbT',\bbT'}^{\bbla'} 
    = H_1 (C^\la_{x_1,x_1} \otimes C^{\la'}_{x_2,x_2})H_1 (C^\la_{y_1,y_1} \otimes C^{\la'}_{y_2,y_2})
    = H_1^2 (C^{\la'}_{x_2,x_2} C^\la_{y_1,y_1} \otimes C^\la_{x_1,x_1} C^{\la'}_{y_2,y_2}).
\end{equation}
On one hand, it follows from {\cite[Lemma 2.2(iii)]{GL96}} that
$C^{\la'}_{x_2,x_2} C^\la_{y_1,y_1}
\in B_{<\la}$ unless $\la' \geq \la$.
By symmetry, $C^\la_{x_1,x_1} C^{\la'}_{y_2,y_2} \in B_{<\la'}$ unless $\la \geq \la'$.
Since $\la \neq \la'$ by the assumption, one of the inequalities must fail and hence
$(C^{\la'}_{x_2,x_2} C^\la_{y_1,y_1} \otimes C^\la_{x_1,x_1} C^{\la'}_{y_2,y_2}) \in B_{<\la}\otimes B_{<\la'} \subseteq \cA_{<\bbla}$.

Suppose that $w_\bbT = w_{\bbT'} = s_1$, and hence 
we may assume that 
$\bbT = (s_1, \tabs_\bbT, (x_1, x_2))$ and $\bbT' = (s_1, \tabs_{\bbT'}, (y_1, y_2))$. Thus,
\begin{equation}
\begin{split}
    C_{\bbT,\bbT}^\bbla C_{\bbT',\bbT'}^{\bbla} 
    &
    = H_1^2 (C^{\la'}_{x_2,x_2} C^{\la'}_{y_2,y_2} \otimes C^\la_{x_1,x_1} C^\la_{y_1,y_1}) H_1^2
    \\
    &=\phi_\la(C^\la_{x_1},C^\la_{y_1})\phi_{\la'}(C^{\la'}_{x_2},C^{\la'}_{y_2})z_m (C^{\la'}_{x_2,y_2} \otimes  C^\la_{x_1,y_1}) H_1^2
    \\
    &= \phi_\la(C^\la_{x_1},C^\la_{y_1})\phi_{\la'}(C^{\la'}_{x_2},C^{\la'}_{y_2}) f_{\la,\la'} H_1 (C^\la_{x_1,y_1}\otimes C^{\la'}_{x_2,y_2}) H_1
    \\
    &= \phi_\la(C^\la_{x_1},C^\la_{y_1})\phi_{\la'}(C^{\la'}_{x_2},C^{\la'}_{y_2}) f_{\la,\la'} C_{\bbT,\bbT'}^\bbla.
\end{split}    
\end{equation}
We omit the case when $w_\bbT = w_{\bbT'} = e$ as it is similar to the case when  $w_\bbT = w_{\bbT'} = s_1$. 
\end{proof}
\begin{cor}\label{rad2}
Suppose that $\bbla = [\la, \la']$ for some $\la \neq \la'$ in $\Pi_m$. 
Then,
\[
\rad(\bbla) \cong  \left( \rad(\phi_\lambda) \otimes S^{\lambda'} \right) \oplus \left( S^\lambda \otimes \rad(\phi_{\lambda'}) \right).
\]
\end{cor}
\begin{proof}
    One has $\rad(\bbla) = \{C^\bbla_\bbT ~|~ \phi_\bbla(C^\bbla_\bbT, C^\bbla_{\bbT'}) = 0 \textup{ for all }\bbT' \in M(\bbla)\}$.
    Since $f_{\la,\la'}\neq 0$, it follows from  \cref{prop:phibbla2} that $ \phi_\bbla(C^\bbla_\bbT, C^\bbla_{\bbT'}) = 0$ if and only if either $C^\la_{\tabt_\bbT^1} \in \rad(\phi_\la)$ or $C^{\la'}_{\tabt_\bbT^2} \in \rad(\phi_{\la'})$. 
    The corollary is proved. 
\end{proof}
\subsection{}
We are now in a position to summarize the information we obtained by having explicit formulas for the canonical bilinear form.

\begin{thm}\label{thm:cellinfo}
    Consider the Hu algebra $\cA(m)$. Then,
    \begin{enua}
        \item  The set of absolutely irreducible $\cA(m)$-modules, up to isomorphisms, is given by
        \[
        \Omega_0 = \{e_\la^\nu \in \Omega ~|~ \la \textup{ is }\ell\textup{-restricted}\} \sqcup 
        \{[\la,\la']\in \Omega ~|~ \la, \la' \textup{ are both }\ell\textup{-restricted}\},
        \]
        where $\ell := \textup{ord}(q)$ is the quantum characteristic of $q\in K$.
        \item  The Hu algebra $\cA(m)$ is semisimple if and only if $\cH_q(\Sigma_m)$ is semisimple.
        \item  The Hu algebra $\cA(m)$ is quasi-hereditary if and only if $\cH_q(\Sigma_m)$ is quasi-hereditary.
    \end{enua}
\end{thm}
\begin{proof}
   The Corollary follows from applying \cref{prop:philainfo}.
For (a), it follows from \cref{prop:phibbla1} that $\phi_{e_\la^\nu}$ is nonzero if and only if $\phi_\la$ is nonzero, which is equivalent to that $\la \vdash m$ is $\ell$-restricted.
On the other hand,
it follows from \cref{prop:phibbla2} that
the form $\phi_{[\la,\la']}$ is nonzero if and only if both $\phi_\la$ and $\phi_{\la'}$ are nonzero.

For (b), by combining \cref{rad1} and \cref{rad2}, one has $\rad(\phi_\bbla) = 0$ for all $\bbla \in \Omega$ if and only if $\rad(\phi_\la) = 0$ for all $\la \vdash m$, which is equivalent to that $\cH_q(\Sigma_m)$ is semisimple.

For (c), by combining \cref{prop:phibbla1} and \cref{prop:phibbla2},  one has that $\phi_{\bbla}\neq 0$ for all $\bbla$ if and only if $\phi_{\la}\neq 0$ for all $\la$.
This completes the proof of the theorem.
\end{proof}
\section{Further Connections}

\subsection{Modular Representations of Hecke Algebras of type $D_{2m}$}

Let $\cH_{(Q,q)}(B_n)$ be the Hecke algebra of type $B_n$ with unequal parameters $(Q, q)$ where $Q = q^b \in K^\times$ for some integer $b$.
When $b$ is sufficiently large, we are in the {\em asymptotic} case (cf. Bonnafe and Iancu \cite{BI03}). 
In this asymptotic regime,
Geck, Iancu, and Pallikaros \cite{gip08} constructed a non-trivial isomorphism between the Specht module $S^{(\la,\mu)}:= S^{\lambda}\otimes S^\mu$
and the cell module $W(\la,\mu)$.
However, there remains a significant and missing special case: the parameter choice $(1, q)$ corresponding to $b=0$. 

The Hecke algebra $\cH^B := \cH_{(1,q)}(B_n)$ is closely related to type $D$, as it contains a subalgebra isomorphic to the Hecke algebra $\cH^D:=\cH_q(D_n)$.
We assume that $n=2m$ is even as the theory can easily recover the odd case.
Suppose that $\ch K \neq 2$ and $q^i \neq -1$ for $1\leq i \leq 2m-1$.
We now outline the findings of Geck \cite{Ge00} and Hu \cite{hu02,Hu09}, regarding the irreducibles for type $D$ using known results from type $B$.

Let $\Pi^{(2)}_{2m}$ be the set of bipartitions of $2m$.
For each $(\lambda, \mu)\in \Pi_i \times \Pi_{2m-i} \subset
\Pi^{(2)}_{2m}$,
denote 
by $S^{(\lambda,\mu)}$ the Specht module of $\cH_q(\Sigma_i \times \Sigma_{2m-i})$, and 
by $V^{(\lambda,\mu)}$ the Specht module over the (split semisimple) Hecke algebra $\cH^B_{(1,v)}$ over the rational function field $\mathbb{Q}(v)$ with unequal parameters $(1,v)$, where $v$ is an indeterminate.

By restricting $V^{(\lambda,\mu)}$ to the type $D$ subalgebra of $\cH^B_{(1,v)}$, 
and subsequently applying the map $\overline{~}$ defined in \cite[Theorem 5.3]{Ge00}, 
Geck obtained the complete list 
$\{\overline{V}^\bbla \}_\bbla$
of non-isomorphic irreducibles, labeled by
\begin{equation}
\label{eq:Dindex}    
\{ (\lambda, +), (\lambda, -) 
~|~ (\lambda,\lambda) \in \Lambda_0\} 
\sqcup
\{ [\lambda, \mu] := \Sigma_2 \cdot (\la,\mu)
~|~ \lambda\neq \mu, \ (\lambda,\mu) \in \Lambda_0\},
\end{equation}
where a bipartition $(\lambda,\mu)$ of $2m$ belongs to $\Lambda_0$ 
if and only if the $\cH_q(\Sigma_i \times \Sigma_{m-i})$-module $D^{(\lambda, \mu)} := S^{(\lambda,\mu)} /\rad S^{(\lambda,\mu)}$ is nonzero.
Observe that in Geck's approach, these simple modules $\overline{V}^\bbla$ are not constructed by quotienting out the radical of a bilinear form. 
Instead, it relies on deciphering the rather intricate map  \cite[Theorem 5.3]{Ge00}.

In \cite{hu02}, Hu observed that one can refine the index set \eqref{eq:Dindex} to the union of $\{ (\lambda, \pm) 
~|~ (\lambda,\lambda) \in \Lambda_0\}$ and $\Pi_m^2\cap \Lambda_0$. 
This subset specifically indexes the simples of the Hu subalgebra $\cala(m)$.
For $\lambda \in \Pi_m$, the induced module $\Ind_{\cH_q(\Sigma_m\times \Sigma_m)}^{\cala(m)}(D^{(\la,\la)})$ splits into non-isomorphic summands denoted by $D^\la_+$ and $D^\la_-$.
Hu further proved that $\overline{V}^{(\la,\pm)} \cong \calf^{-1}(D^\la_\pm)$, where $\calf$ is the Morita equivalence in \eqref{eq:MorD}. 

Our cellularity result gives a more direct realization of the simple $\cH_q(D_{2m})$-modules $\overline{V}^{(\la,\pm)}$ and $\overline{V}^{[\la,\la']}$. In other words:
\begin{cor}
Let $\la, \la' \vdash m$ be such that $\la \neq \la'$.
Then, the  following are isomorphic as simple $\cH_q(D_{2m})$-modules:
\[
\overline{V}^{(\la,+)} \cong \cF^{-1}(L(e_\la^{\iyng{2}})),
\quad
\overline{V}^{(\la,-)} \cong \cF^{-1}(L(e_\la^{\iyng{1,1}})), 
\quad
\textup{and}
\quad
\overline{V}^{[\la,\la']}\cong \cF^{-1}(L([\la,\la'])).
\]
\end{cor}

\subsection{Rational Cherednik Algebras}\label{sec:GGOR}
For a complex reflection group $W$, a family of quasi-hereditary covers of the Hecke algebra $\cH_q(W)$ were constructed by Ginzburg, Guay, Opdam, and Rouquier \cite{GGOR03}.
Their method utilizes the rather implicit Knizhnik-Zamolodchikov functor of the  rational Cherednik algebra $\bbH_W$ corresponding $W$:
\[
\KZ_W: \cO(\bbH(W)) \to \cH_q(W)\-\Mod,
\]
where $\cO(\bbH(W))$ is the associated category $\cO$, which is a highest weight category whose standard objects are denoted by $\Delta^\bbH_\bbla$.
It is natural to seek for a more direct construction of quasi-hereditary covers which generalizes the Schur functor of Dipper--James' $q$-Schur algebra in type A \cite{DJ86}.

For the wreath product $\Sigma_m \wr \Sigma_d$, Nakano, Xiang, and the second named author constructed (1-faithful) quasi-hereditary covers of the corresponding Hecke algebra $\cH_q(\Sigma_m \wr \Sigma_d)$  by utilizing the theory of wreath modules \cite{LNXII}.
In particular, for the case of the Hu algebra $\cH_q(\Sigma_m \wr \Sigma_2) \cong \cala(m)$, the Specht module denoted by $S^\bbla$ therein coincides with the cell modules $W(\bbla)$ with respect to the cellular algebra structure in \cref{thm:A}.

Let $\cS(n,m)$ be the associated Schur algebra for $\cala(m)$ as in \cite{LNX, LNXII}.
The algebra $\cS(n,m)$ is quasi-hereditary and can be thought of as the analog of the $q$-Schur algebra.
The representation theory of $\cS(n,m)$ has been studied therein using homological methods such as spectral sequences and the Schur functor $\Sch : \cS(n,m)$-Mod $\to \cA(m)$-Mod.
In particular, $\cS(n,m)$-Mod is a highest weight category in which the standard module is denoted by $\Delta^\cS_\bbla$, $\bbla \in \Omega$. 
Moreover, under certain conditions which guarantee both $\KZ_D$ and $\Sch$ are 1-faithful quasi-hereditary covers, one has 
\[
\cF \circ \KZ_D(\Delta^\bbH_\bbla) = \Sch(\Delta^\cS_\bbla)
\]
for all $\bbla \in \Omega$.

It is anticipated that one can lift the cellular basis for $\cA(m)$ constructed in this paper to a basis of $\cS(n,m)$ that plays the role of the semistandard tableaux basis of the $q$-Schur algebra. 
Consequently, one obtains an explicit description of the standard modules $\Delta^\cS_\bbla$ in the highest weight category $\cS(n,m)$-Mod. 
\bibliographystyle{alphaabbr}
\bibliography{master}

\end{document}